\documentclass[reqno]{amsart}
\usepackage{mathrsfs}
\usepackage{color}
\usepackage{amsmath}
\usepackage{amsfonts}
\usepackage{amssymb}
\usepackage{graphicx}

 \newtheorem{Theorem}{Theorem}[section]
 \newtheorem{Corollary}[Theorem]{Corollary}
 \newtheorem{Lemma}[Theorem]{Lemma}

 \newtheorem{Definition}[Theorem]{Definition}
\newtheorem{Question}[Theorem]{Question}

 \newtheorem{Remark}[Theorem]{Remark}

 \numberwithin{equation}{section}

\begin{document}

\title[Fiberwise Bergman kernels, vector bundles, and log-subharmonicity]
{Fiberwise Bergman kernels, vector bundles, and log-subharmonicity}
\author{Shijie Bao}
\address{Shijie Bao: Institute of Mathematics, Academy of Mathematics
and Systems Science, Chinese Academy of Sciences, Beijing 100190, China.}
\email{bsjie@amss.ac.cn}

\author{Qi'an Guan}
\address{Qi'an Guan: School of
Mathematical Sciences, Peking University, Beijing 100871, China.}
\email{guanqian@math.pku.edu.cn}


\thanks{}


\date{\today}

\dedicatory{}

\commby{}

\maketitle
\begin{abstract}
In this article, we consider Bergman kernels related to modules at boundary points for singular hermitian metrics on holomorphic vector bundles, and obtain a log-subharmonicity property of the Bergman kernels. As applications, we obtain a lower estimate of weighted $L^2$ integrals on sublevel sets of plurisubharmonic functions, and reprove an effectiveness result of the strong openness property of the modules.
\end{abstract}

\section{Introduction}

It is well-known that the strong openness property of multiplier ideal sheaves (see e.g. \cite{Tian,Nadel,Siu96,DEL,DK01,DemaillySoc,DP03,Lazarsfeld,Siu05,Siu09,DemaillyAG,Guenancia}) has a great influence in the study of several complex variables, complex geometry and complex algebraic geometry
(see e.g. \cite{GZSOC,K16,cao17,cdM17,FoW18,DEL18,ZZ2018,GZ20,berndtsson20,ZZ2019,ZhouZhu20siu's,FoW20,KS20,DEL21}).

Demailly \cite{DemaillySoc,DemaillyAG} conjectured the strong openness property and Guan-Zhou \cite{GZSOC} gave the proof (Jonsson-Musta\c{t}\u{a} \cite{JM12} proved the 2-dimensional case). In order to prove the strong openness property, Jonsson and Musta\c{t}\u{a} (see \cite{JM13}, see also \cite{JM12}) posed the following conjecture, which played an important role in their proof of 2-dimensional strong openness property:

\textbf{Conjecture J-M}: If $c_o^F(\psi)<+\infty$, $\frac{1}{r^2}\mu(\{c_o^F(\psi)\psi-\log|F|<\log r\})$ has a uniform positive lower bound independent of $r\in(0,1)$, where $c_o^{F}(\psi):=\sup\{c\geq 0 : |F|^2e^{-2c\psi}$ is locally $L^1$ near $o \}$, and $\mu$ is the Lebesgue measure.

Guan-Zhou \cite{GZeff} proved Conjecture J-M by using the strong openness property.

Bao-Guan-Yuan \cite{BGY} (see also \cite{GMY-BC2} by Guan-Mi-Yuan) gave an approach to Conjecture J-M independent of the strong
openness property by establishing a concavity property of the minimal $L^{2}$ integrals
with respect to a module at a boundary point of the sub-level sets, and obtained a sharp effectiveness result of Conjecture J-M meanwhile. 

In \cite{BG-BB} (see also \cite{BG-SM}), we considered Bergman kernels related to modules at boundary points of the sub-level sets, and obtained the log-subharmonicity property of the Bergman kernels. We applied the log-subharmonicity to get a lower estimate of weighted $L^2$ integrals on sublevel sets, and reproved the effectiveness result of strong openness property of modules at boundary points.

Recently, for singular hermitian metrics on holomorphic vector bundles, Guan-Mi-Yuan (\cite{GMY5}) established a concavity property of minimal $L^2$ integrals on sublevel sets of plurisubharmonic functions related to modules at boundary points of the sublevel sets, inducing the strong openness property and its effectiveness result of the modules.

It is natural to ask: 

\begin{Question}\label{Q}
Is there an approach from optimal $L^2$ extension theorem to the strong openness property and its effectiveness result related to modules at boundary points for singular hermitian metrics on holomorphic vector bundles?
\end{Question}

In this article, we give an affirmative answer to Question \ref{Q}.

We recall some definitions. Let $M$ be an $n-$dimensional complex manifold. Let $E$ be a rank $r$ holomorphic vector bundle over $M$ and $\overline{E}$ be the conjugate of $E$, $E^*$ be the dual bundle of $E$. Recall that a section $h$ of the vector bundle $E^*\otimes \overline{E}^*$ with measurable coefficients, such that $h$ is an almost everywhere positive definite hermitian form on $E$, is a measurable metric on $E$. And recall that we call a measurable metric $\hat{h}$ on $E$ has a positive locally lower bound if for any compact subset $K$ of $M$, there exists a constant $C_K>0$ such that $\hat{h}>C_K h_1$ on $K$, where $h_1$ is a smooth metric on $E$.

Then we recall the following definition of singular hermitian metrics on vector bundles. 
\begin{Definition}[see \cite{GMY5}]
	Let $M$, $E$ and $h$ be as above and $\Sigma\subset M$ be a closed set of measure zero. Let $\{M_j\}_{j=1}^{+\infty}$ be a sequence of relatively compact subsets of $M$ such that $M_1\subset\subset M_2\subset\subset\ldots\subset\subset M_j\subset\subset M_{j+1}\subset\subset\ldots$ and $\bigcup_{j=1}^{+\infty}M_j=M$. Assume that for each $M_j$, there exists a sequence of hermitian metrics $\{h_{j,s}\}_{s=1}^{+\infty}$ on $M_j$ of class $C^2$ such that $\lim_{s\rightarrow+\infty}h_{j,s}=h$ point-wisely on $M_j\setminus\Sigma$. Then the collection of data $(M,E,\Sigma,M_j,h,h_{j,s})$ is called a singular hermitian metric on $E$.
\end{Definition}

Next we recall the following singular version of Nakano positivity. Let $ D$ be a hermitian metric on $M$, $\theta$ be a hermitian form on $TM$ with continuous coefficients, and $(M,E,\Sigma,M_j,h,h_{j,s})$ be a singular hermitian metric on $E$.
\begin{Definition}[see \cite{GMY5}]\label{singular nak}
	We write:
	\[\Theta_h(E)\geq^s_{Nak}\theta\otimes Id_E\]
	if the following requirements are met.
	
	For each $M_j$, there exists a sequence of continuous functions $\lambda_{j,s}$ on $\overline{M_j}$ and a continuous function $\lambda_j$ on $\overline{M_j}$ subject to the following requirements:
	
	(1) for any $x\in M_j$: $|e_x|_{h_{j,s}}\leq|e_x|_{h_{j,s+1}}$ for any $s\in\mathbb{N}$ and any $e_x\in E_x$;
	
	(2) $\Theta_{h_{j,s}}(E)\geq_{Nak}\theta-\lambda_{j,s} \omega\otimes Id_E$ on $M_j$;
	
	(3) $\lambda_{j,s}\rightarrow 0$ a.e. on $M_j$;
	
	(4) $0\leq\lambda_{j,s}\leq\lambda_j$ on $M_j$ for any $s$.
	
\end{Definition}

\subsection{Main result}\label{main}
\
 Let $M$ be an $n-$dimensional Stein manifold. Let $K_M$ be the canonical line bundle on $M$. Let $dV_M$ be a continuous volume form on $M$. Let $\psi$ be a plurisubharmonic function on $M$. Let $F\not\equiv 0$ be a holomorphic function on $M$, and let $T\in [-\infty,+\infty)$. Denote that
\[\Psi:=\min\{\psi-2\log |F|,-T\}.\]
If $F(z)=0$ for some $z\in M$, set $\Psi(z)=-T$. Let $E$ be a holomorphic vector bundle on $M$ with rank $r$.

Let $(V,z)$ be a local coordinate near a point $z_0$ of $M$ and $E|_V$ is trivial. Then for any $g\in H^0(V,\mathcal{O}(K_M\otimes E))$, there exists a holomorphic $(n,0)$ form $\hat{g}$ on $V$ such that $g=\hat{g}\otimes e$ locally, where $e$ is a local section of $E$ on $V$. Denote that $|g|_{h_0}^2|_V:=\sqrt{-1}^{n^2}g\wedge \bar{g}\langle e,e\rangle_{h_0}$, where $h_0$ is any (smooth or singular) metric on $E$. It can be checked that $|g|_{h_0}^2|_V$ is invariant under the coordinate change and $|g|_{h_0}^2$ is a globally defined $(n,n)$ form on $V$.

Note that for any $t\geq T$, $M_t=\{\psi+2\log|1/F|<-t\}$ on $M\setminus\{F=0\}$. Hence $M_t$ is a Stein submanifold of $M$ for any $t\geq T$ (see \cite{FN80}) , and $\Psi=\psi+2\log|1/F|$ is a plurisubharmonic function on $M_t$.

Let $\hat{h}$ be a smooth metric on $E$. Let $h$ be a measurable metric on $E$ satisfying that $h$ has a positive locally lower bound. Assume $(M,E,\Sigma, M_j, h,h_{j,s})$ is a singular metric on $E$, and $\Theta_{h}(E)\geq^s_{Nak} 0$. 

For any $t\geq T$, denote that
\[A^2(M_t,h):=\{f\in H^0(M_t,\mathcal{O}(K_M\otimes E)) : \int_{M_t}|f|_h^2<+\infty\}.\]
For any $t\in [T,+\infty)$ and $\lambda>0$, denote that
\[\Psi_{\lambda,t}:=\lambda\max\{\Psi+t,0\}.\]
And for any $f\in A^2(M_T,h)$, denote that
\[\|f\|_{\lambda,t}:=\left(\int_{M_T}|f|_h^2e^{-\Psi_{\lambda,t}}\right)^{1/2}.\]
Note that
\[\|f\|_T^2:=\|f\|^2_{\lambda,T}=\int_{M_T}|f|_h^2\]
for any $\lambda>0$, and
\[e^{\lambda(T-t)}\|f\|_T^2\leq\|f\|^2_{\lambda,t}\leq\|f\|_T^2<+\infty\]
for any $t\geq T$.

We will state that $A^2(M_T,h)$ is a Hilbert space in Section 2. Denote the dual space of $A^2(M_T,h)$ by $A^2(M_T,h)^*$. For any $\xi\in A^2(M_T,h)^*$, denote that the Bergman kernel with respect to $\xi$ is
\[K^{h}_{\xi,\Psi,\lambda}(t):=\sup_{f\in A^2(M_T,h)}\frac{|\xi\cdot f|^2}{\|f\|^2_{\lambda,t}}\]
for any $t\in [T,+\infty)$, where $K^{h}_{\xi,\Psi,\lambda}(t)=0$ if $A^2(M_T,h)=\{0\}$.

Denote $U_T:=(T,+\infty)+\sqrt{-1}\mathbb{R}:=\{w\in\mathbb{C} : \text{Re\ }w>T\}\subset\mathbb{C}$. We obtain the following log-subharmonicity property of the Bergman kernel $K^{h}_{\xi,\Psi,\lambda}$.
\begin{Theorem}\label{concavity}
	Assume that $A^2(M_T,h)\neq \{0\}$. Then $\log K^{h}_{\xi,\Psi,\lambda}(\text{Re\ }w)$ is subharmonic with respect to $w\in U_T$.
\end{Theorem}

When $F\equiv 1$, we have $\Psi\equiv\psi$ on $\{\psi<-T\}$, and Theorem \ref{concavity} induces the following corollary related to fiberwise Bergman kernels with respect to plurisubharmonic functions.
\begin{Corollary}
	Assume that $K^{h}_{\xi,\psi,\lambda}(t_0)\in (0,+\infty)$ for some $t_0\geq T$. Then $\log K^{h}_{\xi,\psi,\lambda}(\text{Re\ }w)$ is subharmonic with respect to $w\in U_T$.
\end{Corollary}

We recall some notations in \cite{GMY5}. Let $z_0$ be a point in $M$. Denote that
\[\tilde{J}(E,\Psi)_{z_0}:=\{f\in H^0(\{\Psi<-t\}\cap V,\mathcal{O}(E)) : t\in\mathbb{R},\ V \text{\ is \ a \ neighborhood \ of \ } z_0\},\]
and
\[J(E,\Psi)_{z_0}:=\tilde{J}(E,\Psi)_{z_0}/\sim,\]
where the equivalence relation `$\sim$' is as follows:
\[f \sim g \ \Leftrightarrow \ f=g \text{\ on \ } \{\Psi<-t\}\cap V, \text{\ where\ }  t\gg T, V \text{\ is\ a\ neighborhood\ of\ } z_0.\]
For any $f\in \tilde{J}(E,\Psi)_{z_0}$, denote the equivalence class of $f$ in $J(E,\Psi)_{z_0}$ by $f_{z_0}$. And for any $f_{z_0},g_{z_0}\in J(E,\Psi)_{z_0}$, and $(q,z_0)\in\mathcal{O}_{M,z_0}$, define
\[f_{z_0}+g_{z_0}:=(f+g)_{z_0},\ (q,z_0)\cdot f_{z_0}:=(qf)_{z_0}.\]
It is clear that $J(E,\Psi)_{z_0}$ is an $\mathcal{O}_{M,z_0}-$module.

For any $a\geq 0$, denote that $I(h,a\Psi)_{z_0}:=\big\{f_{z_0}\in J(E,\Psi)_{z_0} : \exists t\gg T, V$ is a neighborhood of $z_0,\ \text{s.t.\ } \int_{\{\Psi<-t\}\cap V}|f|_h^2e^{-a\Psi}dV_M<+\infty\big\}$, where $dV_M$ is a continuous volume form on $M$. Then it is clear that $I(h,a\Psi)_{z_0}$ is an $\mathcal{O}_{M,z_0}-$submodule of $J(E,\Psi)_{z_0}$. Especially, we denote that $I_{z_0}:=I(\hat{h},0\Psi)_{z_0}$, where $\hat{h}$ is the smooth metric on $E$. If $z_0\in\bigcap_{t>T}\overline{\{\Psi<-t\}}$, then $I_{z_0}=\mathcal{O}(E)_{z_0}$.

Let $Z_0$ be a subset of $\bigcap_{t>T}\overline{\{\Psi<-t\}}$. Let $J_{z_0}$ be an $\mathcal{O}_{M,z_0}-$submodule of $J(E,\Psi)_{z_0}$ for any $z_0\in Z_0$. For any $t\geq T$, denote that
\[A^2(M_t,h)\cap J:=\left\{f\in A^2(M_t,h) :  f_{z_0}\in\mathcal{O}(K_M)_{z_0}\otimes J_{z_0}, \text{for\ any\ } z_0\in Z_0\right\}.\]
Assume that $A^2(M_T,h)\cap J$ is a proper subspace of $A^2(M_T,h)$. Using Theorem \ref{concavity}, we obtain the following concavity and monotonicity property related to $K^{h}_{\xi,\Psi,\lambda}$.

\begin{Theorem}\label{increasing}
	Assume that $A^2(M_T,h)\neq\{0\}$, $I(h,\Psi)_{z_0}\subset J_{z_0}$ for any $z_0\in Z_0$, and $\xi\in A^2(M_T,h)^*$ such that $\xi|_{A^2(M_T,h)\cap J}\equiv 0$. Then $-\log K^{h}_{\xi,\Psi,\lambda}(t)+t$ is concave and increasing with respect to $t\in [T,+\infty)$.
\end{Theorem}

Let $F\equiv 1$, and let $Z_0\subset\{\psi=-\infty\}$. Additionally, we let modules $I(h,\psi)_{z_0}$ and $J_{z_0}$ be ideals of  $\mathcal{O}_{M,z_0}$ for any $z_0\in Z_0$. Then Theorem \ref{increasing} induces the following corollary related to Bergman kernels with respect to interior points.

\begin{Corollary}
	Assume that $A^2(M_T,h)\neq\{0\}$, $I(h,\psi)_{z_0}\subset J_{z_0}$ for any $z_0\in Z_0$, and $\xi\in A^2(M_T,h)^*$ such that $\xi|_{A^2(M_T,h)\cap J}\equiv 0$. Then $-\log K^{h}_{\xi,\psi,\lambda}(t)+t$ is concave and increasing with respect to $t\in [T,+\infty)$.
\end{Corollary}

\subsection{Applications}
\
\
Let $M$ be an $n-$dimensional Stein manifold. Let $K_M$ be the canonical line bundle on $M$. Let $\psi$ be a plurisubharmonic function on $M$. Let $F\not\equiv 0$ be a holomorphic function on $M$, and let $T\in [-\infty,+\infty)$. Denote that
\[\Psi:=\min\{\psi-2\log |F|,-T\}.\]
If $F(z)=0$ for some $z\in M$, set $\Psi(z)=-T$. Let $E$ be a holomorphic vector bundle on $M$ with rank $r$. Let $h$ be a measurable metric on $E$ satisfying that $h$ has a positive locally lower bound. Assume $(M,E,\Sigma, M_j, h,h_{j,s})$ is a singular metric on $E$.  We give the following lower estimate of $L^2$ integrals on sublevel sets $\{\Psi<-t\}$ by Theorem \ref{concavity} and Theorem \ref{increasing}.

\begin{Corollary}[see \cite{GMY5}]\label{L2integral}
	Let $f$ be an $E-$valued holomorphic $(n,0)$ form on $\{\Psi<-t_0\}$ for some $t_0\geq T$ such that $f\in A^2(M_{t_0},h)$. Let $z_0\in M_{t_0}$, and assume that $a_{z_0}^f(\Psi;h)<+\infty$ and $\Theta_{h}(E)\geq^s_{Nak} 0$, where
	\[a_{z_0}^f(\Psi;h):=\sup\{a\geq 0 : f_{z_0}\in \mathcal{O}(K_M)_{z_0}\otimes I(h,2a\Psi)_{z_0}\}.\]
    Then for any $r\in (0,e^{-a_{z_0}^f(\Psi;h)t_0}]$, we have
	\[\frac{1}{r^2}\int_{\{a_{z_0}^f(\Psi;h)\Psi<\log r\}}|f|^2_h\geq e^{2a_{z_0}^f(\Psi;h)t_0}C,\]
	where
	\begin{flalign*}
		\begin{split}
			C:=&C(\Psi,h,I_+(h,2a_{z_0}^f(\Psi;h)\Psi)_{z_0},f,M_{t_0})\\
			:=&\inf\bigg\{\int_{M_{t_0}}|\tilde{f}|_h^2 : \tilde{f}\in A^2(M_{t_0},h)\\ 
			& \& \  (\tilde{f}-f)_{z_0}\in \mathcal{O}(K_M)_{z_0}\otimes I_+(h,2a_{z_0}^f(\Psi;h)\Psi)_{z_0}\bigg\},
		\end{split}
	\end{flalign*} 
	and
	\[I_+(h,p\Psi)_{z_0}=\bigcup_{p'>p}I(h,p'\Psi)_{z_0}\]
	for any $p>0$.
\end{Corollary}

\begin{Remark}\label{a>0}
	In Corollary \ref{L2integral}, for any $z_0\in M$, the proof of $a_{z_0}^f(\Psi;h)>0$ can be referred to \cite{GMY5}.
\end{Remark}

When $F\equiv 1$, Corollary \ref{L2integral} gives a lower estimate of $L^2$ integrals on sublevel sets of plurisubharmonic function.

\begin{Corollary}
	Let $f$ be an $E-$valued holomorphic $(n,0)$ form on $\{\psi<-t_0\}$ for some $t_0\geq T$ such that $f\in A^2(M_{t_0},h)$. Let $z_0\in M_{t_0}$, and assume that $a_{z_0}^f(\psi;h)<+\infty$ and $\Theta_{h}(E)\geq^s_{Nak} 0$, where
	\[a_{z_0}^f(\psi;h):=\sup\{a\geq 0 : f_{z_0}\in \mathcal{O}(K_M)_{z_0}\otimes I(h,2a\psi)_{z_0}\}.\]
	 Then for any $r\in (0,e^{-a_{z_0}^f(\psi;h)t_0}]$, we have
	\[\frac{1}{r^2}\int_{\{a_{z_0}^f(\psi;h)\psi<\log r\}}|f|^2_h\geq e^{2a_{z_0}^f(\psi;h)t_0}C,\]
	where
	\begin{flalign*}
		\begin{split}
			C:=&C(\psi,h,I_+(h,2a_{z_0}^f(\psi;h)\psi)_{z_0},f,M_{t_0})\\
			:=&\inf\bigg\{\int_{M_{t_0}}|\tilde{f}|_h^2 : \tilde{f}\in A^2(M_{t_0},h)\\ 
			& \& \  (\tilde{f}-f)_{z_0}\in \mathcal{O}(K_M)_{z_0}\otimes I_+(h,2a_{z_0}^f(\psi;h)\psi)_{z_0}\bigg\},
		\end{split}
	\end{flalign*} 
	and
	\[I_+(h,p\psi)_{z_0}=\bigcup_{p'>p}I(h,p'\psi)_{z_0}\]
	for any $p>0$.
\end{Corollary}

Theorem \ref{concavity} and Theorem \ref{increasing} also deduce a reproof of the following effectiveness result of strong openness property of the module $I(h,a\Psi)_{z_0}$ on vector bundles.

\begin{Corollary}[see \cite{GMY5}]\label{SOPE}
	Let $f$ be a holomorphic $(n,0)$ form on $M_{t_0}=\{\Psi<-t_0\}$ for some $t_0\geq T$ such that $f\in A^2(M_{t_0},h)$. Let $z_0\in M$. Assume that $a_{z_0}^f(\Psi;h)<+\infty$ and $\Theta_{h}(E)\geq^s_{Nak} 0$. Let $C_1$ and $C_2$ be two positive constants. If
	
	(1) $\int_{M_{t_0}}|f|_h^2e^{-\Psi}\leq C_1$;
	
	(2) $C(\Psi,h,I_+(h,2a_{z_0}^f(\Psi;h)\Psi)_{z_0},f,M_{t_0})\geq C_2$,\\	
	then for any $q>1$ satisfying
	\[\theta(q)>\frac{C_1}{C_2},\]
	we have $f_{z_0}\in \mathcal{O}(K_M)_{z_0}\otimes I(h,q\Psi)_{z_0}$, where $\theta(q)=\frac{q}{q-1}e^{t_0}$.
\end{Corollary}

For $F\equiv 1$, Corollary \ref{SOPE} degenerates to the effectiveness result of strong openness property with respect to interior points.

\begin{Corollary}
	Let $f$ be a holomorphic $(n,0)$ form on $M_{t_0}=\{\psi<-t_0\}$ for some $t_0\geq T$ such that $f\in A^2(M_{t_0},h)$. Let $z_0\in M$. Assume that $a_{z_0}^f(\psi;h)<+\infty$ and $\Theta_{h}(E)\geq^s_{Nak} 0$. Let $C_1$ and $C_2$ be two positive constants. If
	
	(1) $\int_{M_{t_0}}|f|_h^2e^{-\psi}\leq C_1$;
	
	(2) $C(\psi,h,I_+(h,2a_{z_0}^f(\psi;h)\psi)_{z_0},f,M_{t_0})\geq C_2$,\\	
	then for any $q>1$ satisfying
	\[\theta(q)>\frac{C_1}{C_2},\]
	we have $f_{z_0}\in \mathcal{O}(K_M)_{z_0}\otimes I(h,q\psi)_{z_0}$, where $\theta(q)=\frac{q}{q-1}e^{t_0}$.
\end{Corollary}

\section{Preparations}
\subsection{$L^2$ methods}
\

We need the following optimal $L^2$ extension theorem, which can be referred to \cite{GMY-L2ext}. And for the convenience of readers, we give a proof in appendix.

Let $M$ be an $n-$dimensional Stein manifold. Let $D=\Delta_{w_0,r}=\{w\in\mathbb{C} : |w-w_0|<r\}\subset U_T$, where $w_0\in U_T$, $r>0$, and $w$ is the coordinate on $D$. Let $\Omega:=M\times D$ be an $(n+1)-$dimensional complex manifold, and $p_1,p_2$ be the natural projections from $\Omega$ to $D$ and $M$. Let $E$ be a holomorphic vector bundle on $M$ with rank $r$. Let $h$ be a measurable metric on $E$ satisfying that $h$ has a positive locally lower bound. Assume $(M,E,\Sigma, M_j, h,h_{j,s})$ is a singular metric on $E$, and $\Theta_{h}(E)\geq^s_{Nak} 0$. 

Let $E':=p_2^*(E)$ be a vector bundle over $\Omega$. Then $p_2^*(h)$ is a measurable metric on $E'$ induced by the construction of $E'$. It can be checked that $p_2^*(h)$ has a positive locally lower bound on $E'$, $(\Omega,E',\Sigma\times D, M_j\times D, p_2^*(h),p_2^*(h_{j,s}))$ is a singular metric on $E'$, and $\Theta_{p_2^*(h)}(E')\geq_{Nak}^s 0$.

Let $\tilde{\Psi}$ be a bounded plurisubharmonic function on $\Omega$. Denote that $\tilde{\Psi}_{w}:=\tilde{\Psi}|_{M\times\{w\}}$.

\begin{Lemma}\label{L2ext}
For any $E-$valued holomorphic $(n,0)$ form $u$ on $M$ such that $\int_{M}|u|_h^2e^{-\tilde{\Psi}_{w_0}}<+\infty$, there exists an $E'-$valued holomorphic $(n+1,0)$ form $\tilde{u}$ on $\Omega$, such that $\tilde{u}=u\wedge dw$ on $M\times\{w_0\}$, and
\[\frac{1}{\pi r^2}\int_{\Omega}|\tilde{u}|^2_{p_2^*(h)}e^{-\tilde{\Psi}}\leq\int_{M}|u|_h^2e^{-\tilde{\Psi}_{w_0}}.\]
\end{Lemma}

Let $M$ be an $n-$dimensional Stein manifold. Let $F\not\equiv 0$ be a holomorphic function on $M$, and $\psi$ be a plurisubharmonic function on $M$. Let $E$ be a holomorphic vector bundle on $M$ with rank $r$. Let $h$ be a measurable metric on $E$ satisfying that $h$ has a positive locally lower bound. Denote that $\tilde{h}:=he^{-\psi}$. Let $(M,E,\Sigma, M_j, \tilde{h},\tilde{h}_{j,s})$ be a singular metric on $E$, and assume that $\Theta_{\tilde{h}}(E)\geq^s_{Nak} 0$. Let $T$ be a real number. Denote that
\[\tilde{\varphi}:=\max\{\psi+T,2\log|F|\},\]
and
\[\Psi:=\min\{\psi-2\log|F|, -T\}.\]
If $F(z)=0$ for some $z\in M$, set $\Psi(z)=-T$. The following lemma will be used to prove Theorem \ref{increasing}.

\begin{Lemma}[\cite{GMY5}]\label{L2mthod}
	Let $t_0\in (T,+\infty)$ be arbitrary given. Let $f$ be an $E-$valued holomorphic $(n,0)$ form on $\{\Psi<-t_0\}$ such that
	\[\int_{\{\Psi<-t_0\}}|f|_h^2<+\infty.\]
	Then there exists an $E-$valued holomorphic $(n,0)$ form $\tilde{F}$ on $M$ such that
	\[\int_M|\tilde{F}-(1-b_{t_0}(\Psi))fF^2|^2_{\tilde{h}}e^{v_{t_0}(\Psi)-\tilde{\varphi}}\leq C\int_M\mathbb{I}_{\{-t_0-1<\Psi<-t_0\}}|fF|_{\tilde{h}}^2,\]
	where $b_{t_0}(t)=\int_{-\infty}^t\mathbb{I}_{\{-t_0-1<s<-t_0\}}\mathrm{d}s$, $v_{t_0}(t)=\int_{-t_0}^tb_{t_0}(s)\mathrm{d}s-t_0$ and $C$ is a positive constant independent of $t_0$ and $f$.
\end{Lemma}

\subsection{Some lemmas about submodules of $J(E,\Psi)$}
\
Let $F$ be a holomorphic function on a pseudoconvex domain $D\subset\mathbb{C}^n$ containing the origin $o\in\mathbb{C}^n$. Let $\psi$ be a plurisubharmonic function on $D$. Let $f=(f_1,\ldots,f_r)$ be a holomorphic section of $E:=D\times \mathbb{C}^r$. Let $h$ be a measurable metric on $E$ satisfying that $h$ has a positive locally lower bound. Assume $(D,E,\Sigma, D_j, h,h_{j,s})$ is a singular metric on $E$, and $\Theta_{h}(E)\geq^s_{Nak} 0$. Let $T$ be a real number. Denote that
\[\Psi:=\min\{\psi-2\log|F|,-T\}.\]
If $F(z)=0$ for some $z\in D$, we set $\Psi(z)=-T$.

We recall the following lemma.

\begin{Lemma}[\cite{GMY5}]\label{l:converge}
	Let $J_o$ be an $\mathcal{O}_{\mathbb{C}^n,o}-$submodule of $I(h,0\Psi)_o$ such that $I(h,\Psi)_o\subset J_o$. Assume that $f_o\in J(E,\Psi)_o$. Let $U_0$ be a Stein open neighborhood of $o$. Let $\{f_j\}_{j\geq 1}$ be a sequence of $E-$valued holomorphic $(n,0)$ forms on $U_0\cap\{\Psi<-t_j\}$ for any $j\geq 1$, where $t_j\in (T,+\infty)$. Assume that $t_0=\lim_{j\rightarrow+\infty}t_j\in [T,+\infty)$,
	\[\limsup_{j\rightarrow+\infty}\int_{U_0\cap\{\Psi<-t_j\}}|f_j|_h^2\leq C<+\infty,\]
	and $(f_j-f)_o\in J_o$. Then there exists a subsequence of $\{f_j\}_{j\geq 1}$ compactly convergent to an $E-$valued holomorphic $(n,0)$ form $f_0$ on $\{\Psi<-t_0\}\cap U_0$ which satisfies
	\[\int_{U_0\cap\{\Psi<-t_0\}}|f_0|_h^2\leq C,\]
	and $(f_0-f)_o\in J_o$.
\end{Lemma}

Let $M$ be an $n-$dimensional complex manifold. Let $K_M$ be the canonical line bundle on $M$. Let $\psi$ be a plurisubharmonic function on $M$. Let $F\not\equiv 0$ be a holomorphic function on $M$, and let $T\in [-\infty,+\infty)$. Denote that
\[\Psi:=\min\{\psi-2\log |F|,-T\}.\]
If $F(z)=0$ for some $z\in M$, set $\Psi(z)=-T$. Let $E$ be a holomorphic vector bundle on $M$ with rank $r$. Let $h$ be a measurable metric on $E$ satisfying that $h$ has a positive locally lower bound. Assume $(M,E,\Sigma, M_j, h,h_{j,s})$ is a singular metric on $E$, and $\Theta_{h}(E)\geq^s_{Nak} 0$.

Recall that
\[A^2(M_t,h):=\{f\in H^0(M_t,\mathcal{O}(K_M\otimes E)) : \int_{M_t}|f|_h^2<+\infty\}\]
for any $t\geq T$. Let $Z_0$ be a subset of $M$. Let $J_{z_0}$ be an $\mathcal{O}_{M,z_0}-$submodule of $J(E,\Psi)_{z_0}$ for any $z_0\in Z_0$. For any $t\geq T$, denote that
\[A^2(M_t,h)\cap J:=\left\{f\in A^2(M_t,h) :  f_{z_0}\in\mathcal{O}(K_M)_{z_0}\otimes J_{z_0}, \text{for\ any\ } z_0\in Z_0\right\}.\]
We state that $A^2(M_T,h)\cap J$ is a closed subspace of $A^2(M_T,h)$ if $J_{z_0}\supset I(h,\Psi)_{z_0}$ for any $z_0\in Z_0$.

\begin{Lemma}\label{Jclosed}
	Assume that $J_{z_0}\supset I(h, \Psi)_{z_0}$ for any $z_0\in Z_0$. Then $A^2(M_T,h)\cap J$ is closed in $A^2(M_T,h)$.
\end{Lemma}

\begin{proof}
	Let $\{f_j\}$ be a sequence of $E-$valued holomorphic $(n,0)$ forms in $A^2(M_T,h)\cap J$, such that $\{f_j\}$ is a Cauchy sequence under the topology of $A^2(M_T,h)$. Then $\int_{M_T}|f_j|^2_h$ is uniformly bounded. Using Lemma \ref{l:converge} and diagonal method, for any subsequence $\{f_{k_j}\}$ of $\{f_j\}$, we can find a further subsequence compactly convergent to an $E-$valued holomorphic $(n,0)$ form $f_0$ on $M_T$. With Fatou's Lemma, we have
	\[\int_{M_T}|f_0|^2_h\leq \liminf_{j\to+\infty}\int_{M_T}|f_{k_j}|^2_h<+\infty,\]
	which means that $f_0\in A^2(M_T,h)$. For any $\epsilon>0$, there exists $N>0$ such that for any $m,n>N$, we have
	\[\int_{M_T}|f_m-f_n|^2_h<\epsilon.\]
	Then for any $m>N$, it follows from Fatou's Lemma that
	\begin{equation*}
	       \int_{M_T}|f_m-f_0|^2_h\leq\liminf_{j\to+\infty}\int_{M_T}|f_m-f_{k_j}|^2_h\leq\epsilon.
	\end{equation*}
	This shows that $\{f_j\}$ converges to $f_0$ under the topology of $A^2(M_T,h)$.
	
	Note that $(f_j)_{z_0}\in \mathcal{O}(K_M)_{z_0}\otimes J_{z_0}$ for any $j$ and $z_0\in Z_0$. According to Lemma \ref{l:converge}, we can get that $(f_0)_{z_0}\in \mathcal{O}(K_M)_{z_0}\otimes J_{z_0}$ for any $z_0\in Z_0$, which means that $f_0\in A^2(M_T,h)\cap J$. The we know that $A^2(M_T,h)\cap J$ is closed in $A^2(M_T,h)$.
\end{proof}

Note that when $Z_0=\emptyset$ (or $J_{z_0}=J(E,\Psi)_{z_0}$ for any $z_0\in Z_0$), Lemma \ref{Jclosed} implies that $A^2(M_T,h)$ is a Hilbert space.  

\begin{Corollary}
    $A^2(M_T,h)$ is a Hilbert space.
\end{Corollary}

\subsection{Some lemmas about functionals on $A^2(M_T,h)$}
\

The following two lemmas will be used in the proof of Theorem \ref{concavity}.

Let $M$ be an $n-$dimensional complex manifold. Let $E$ be a holomorphic vector bundle on $M$ with rank $r$. Let $\hat{h}$ be a smooth metric on $E$. Let $h$ be a measurable metric on $E$ satisfying that $h$ has a positive locally lower bound. Assume $(M,E,\Sigma, M_j, h,h_{j,s})$ is a singular metric on $E$, and $\Theta_{h}(E)\geq^s_{Nak} 0$.

\begin{Lemma}\label{fjtof0}
	Let $\{f_j\}$ be a sequence in $A^2(M,h)$, such that $\int_M|f_j|_h^2$ is uniformly bounded for any $j\in\mathbb{N}_+$. Assume that $f_j$ compactly converges to $f_0\in A^2(M,h)$. Then for any $\xi\in A^2(M,h)^*$,
	\[\lim_{j\rightarrow+\infty}\xi\cdot f_j=\xi\cdot f_0.\]
\end{Lemma}

\begin{proof}
	For any $f\in A^2(M,h)$, denote that $\|f\|^2:=\int_M|f|^2_h$. Let $\{f_{k_j}\}$ be any subsequence of $\{f_j\}$. Since $A^2(M,h)$ is a Hilbert space, and $\|f_{k_j}\|^2$ is uniformly bounded, there exists a subsequence of $\{f_{k_j}\}$ (denoted by $\{f_{k_{l_j}}\}$) weakly convergent to some $\tilde{f}\in A^2(M,h)$. 
	
	Let $\{U_l\}$ be an open cover of the complex manifold $M$, such that $E|_{U_l}$ is trivial. Let $(U_l,w_l)$ be the local coordinate on each $U_l$, and $e_l=(e_{l,1},\ldots,e_{l,r})$ is a local section of $E$ on $U_l$. Then we may denote that $f_j=\sum_{k=1}^r f_{j,l,k}dw_l\otimes e_{l,k}$, $f_0=\sum_{k=1}^r g_{0,l,k}dw_l\otimes e_{l,k}$, and $\tilde{f}=\sum_{k=1}^r\tilde{g}_{l,k}dw_l\otimes e_{l,k}$ on each $U_l$, where $f_{j,l,k},g_{0,l,k}$ and $\tilde{g}_{l,k}$ are holomorphic functions on $U_l$. For any $z\in M$, denote that $S_z:=\{l : z\in U_l\}$. For any $l\in S_z$ and $k\in\{1,\ldots,r\}$, let $\xi_{z,l,k}$ be the functional defined as follows:
	\begin{flalign*}
		\begin{split}
			\xi_{z,l,k} \ : \ A^2(M,h)&\longrightarrow\mathbb{C}\\
			f&\longmapsto f_{l,k}(z),
		\end{split}
	\end{flalign*}
	where $f=\sum_{k=1}^rf_{l,k}dw_l\otimes e_{l,k}$ on $U_l$, and $f_{l,k}$ is a holomorphic function on $U_l$. It is clear that the functional $\xi_{z,l,k}\in A^2(M,h)^*$ for any $z\in M$, $l\in S_z$ and $k\in\{1,\ldots,r\}$, since $h$ has a positive locally lower bound. Then we have
	\[g_{0,l,k}(z)=\lim_{j\rightarrow+\infty}\xi_{z,l,k}\cdot f_j=\lim_{j\rightarrow+\infty}\xi_{z,l,k}\cdot f_{k_{l_j}}=\xi_{z,l}\cdot\tilde{f}=\tilde{g}_{l,k}(z), \ \forall z\in M, l\in S_z, 1\leq k\leq r,\]
	thus $f_0=\tilde{f}$. It means that $\{f_{k_j}\}$ has a subsequence weakly convergent to $f_0$. Since $\{f_{k_j}\}$ is an arbitrary subsequence of $\{f_j\}$, we get that $\{f_j\}$ weakly converges to $f_0$. In other words, for any $\xi\in A^2(M,h)^*$,
	\[\lim_{j\rightarrow+\infty}\xi\cdot f_j=\xi\cdot f_0.\]
\end{proof}

Let $\Omega:=M\times D$, where $M$ is an $n-$dimensional complex manifold, and $D$ is a domain in $\mathbb{C}$. Let $E$ be a holomorphic vector bundle on $M$ with rank $r$. Let $E':=E\boxtimes (D\times\mathbb{C})$ be a holomorphic vector bundle on $\Omega$, here $D\times\mathbb{C}$ is the trivial line bundle on $D$. Let $h$ be a measurable metric on $E$ satisfying that $h$ has a positive locally lower bound. Assume $(M,E,\Sigma, M_j, h,h_{j,s})$ is a singular metric on $E$, and $\Theta_{h}(E)\geq^s_{Nak} 0$. Let $f$ be an $E'-$valued holomorphic $(n+1,0)$ form on $\Omega$. For any $\tau\in D$, denote that
\[f_{\tau}:=\frac{f}{d\tau}|_{M_{\tau}}\]
is an $E'-$valued holomorphic $(n,0)$ form on $M_{\tau}$, where $M_{\tau}:=\pi_2^{-1}(\tau)$, and $\pi_2$ is the natural projection from $\Omega$ to $D$. Assume that
\[\int_D\left(\int_{M_{\tau}}|f_{\tau}|^2_h\right)d\lambda_D<+\infty,\]
where $\lambda_D$ is the Lebesgue measure on $D$.

\begin{Lemma}\label{xihol}
	For any $\xi\in A^2(M,h)^*$, $\xi\cdot f_{\tau}$ is holomorphic with respect to $\tau\in D$.
\end{Lemma}

\begin{proof}
	We only need to prove that $h(\tau):=\xi\cdot f_{\tau}$ is holomorphic near any $\tau_0\in D$. Since $\tau_0\in  D$, there exists $r>0$ such that $\Delta(\tau_0,2r)\subset\subset D$. Then for any $\tau\in\Delta(\tau_0,r)$, according to sub-mean value inequality of subharmonic functions, we have
	\[\int_M|f_{\tau}|^2_h\leq \frac{1}{\pi r^2}\int_{\Delta(\tau,r)}\left(\int_M|f|^2_h\right)d\lambda_D<+\infty,\]
	which implies that $f_{\tau}\in A^2(M,h)$ and there exists $C>0$ such that $\int_M|f_{\tau}|^2_h\leq C$ for any $\tau\in\Delta(\tau_0,r)$.
	
	Let $\{U_l\}$ be an open cover of the complex manifold $M$, and $(U_l,w_l)$ be the local coordinate on each $U_l$.
	For any $z\in M$, Denote that $S_z:=\{l : z\in U_l\}$. And for any $l\in S_z$, $k\in\{1,\ldots,r\}$, let $\xi_{z,l,k}$ be the functional in the proof of Lemma \ref{fjtof0}. In the Hilbert space $A^2(M,h)$, by Riesz representation theorem, there exists $\phi_{z,l,k}\in A^2(M,h)$ such that
	\[\xi_{z,l,k}\cdot g=\sqrt{-1}^{n^2}\int_M \langle g, \phi_{z,l,k}\rangle_h\]
	for any $z\in M$, $l\in S_z$, $k\in\{1,\ldots,r\}$. Denote that
	\[H:=\overline{\text{span}\{\phi_{z,l,k} : z\in M, l\in S_z, 1\leq k\leq r\}}\]
	is a closed subspace of $A^2(M,h)$. If $H\neq A^2(M,h)$, then the closed subspace $H^{\bot}\neq\{0\}$. Choosing some $g_0\in H^{\bot}$ with $g_0\neq 0$, we have that for any $z\in M$, $l\in S_z$, and $k\in\{1,\ldots,r\}$,
	$\xi_{z,l,k}\cdot g_0=0$ holds. Then it is clear that $g_0=0$, which is a contradiction. Thus $H=A^2(M,h)$. Denote that
	\[L:=\text{span}\{\xi_{z,l,k} : z\in M, l\in S_z,1\leq k\leq r\}\subset A^2(M,h)^*.\]
	Since $H=A^2(M,h)$, we can find a sequence $\{\xi_j\}\subset L\subset A^2(M,h)^*$, such that
	\[\lim_{k\rightarrow+\infty}\|\xi_j-\xi\|_{A^2(M,h)^*}=0.\]
	It is clear that for any $z\in M$, $l\in S_z$ and $k\in\{1,\ldots,r\}$, $\xi_{z,l,k}\cdot f_{\tau}$ is holomorphic with respect to $\tau\in D$. Then for any $k$, $h_j(\tau):=\xi_j\cdot f_{\tau}$ is holomorphic with respect to $\tau\in D$.
	Besides, for any $\tau\in\Delta(\tau_0,r)$, we have
	\begin{flalign*}
		\begin{split}
			&|h_j(\tau)-h(\tau)|^2\\
			=&|(\xi_j-\xi)\cdot f_{\tau}|^2\\
			\leq&\|\xi_j-\xi\|^2_{A^2(M,h)^*}\int_M|f_{\tau}|^2_h\\
			\leq&C\|\xi_j-\xi\|^2_{A^2(M,h)^*},
		\end{split}
	\end{flalign*}
	which means that $h_j$ uniformly converges to $h$ on $\Delta(\tau_0,r)$. According to Weierstrass theorem, we know that $h$ is holomorphic on $\Delta(\tau_0,r)$, i.e. near $\tau_0$. Then we get that $\xi\cdot f_{\tau}$ is holomorphic with respect to $\tau\in D$.
\end{proof}

\subsection{Some properties of $K^{h}_{\xi,\Psi,\lambda}(t)$}
\

In this section, we prove some properties of the Bergman kernel $K^{h}_{\xi,\Psi,\lambda}(t)$.

Let $\xi\in A^2(M_T,h)^*\setminus\{0\}$. We need the following lemma.

\begin{Lemma}\label{sup=max}
	For any $t\in [T,+\infty)$, if $K^{h}_{\xi,\Psi,\lambda}(t)\in (0,+\infty)$, then there exists $\tilde{f}\in A^2(M_T,h)$, such that
	\[K^{h}_{\xi,\Psi,\lambda}(t)=\frac{|\xi\cdot \tilde{f}|^2}{\|\tilde{f}\|_{\lambda,t}^2}.\]
\end{Lemma}

\begin{proof}
	By the definition of $K^{h}_{\xi,\Psi,\lambda}(t)$, there exists a sequence $\{f_j\}$ of $E-$valued holomorphic $(n,0)$ forms in $A^2(M_T,h)$, such that $\|f_j\|_{\lambda,t}=1$, and $\lim_{j\rightarrow+\infty}|\xi\cdot f_j|^2=K^{h}_{\xi,\Psi,\lambda}(t)$. Then $\int_{M_T}|f_j|^2_h$ is uniformly bounded. Following from Montel's theorem, we can get a subsequence of $\{f_j\}$ compactly convergent to an $E-$valued holomorphic $(n,0)$ form $\tilde{f}$ on $M_T$. According to Fatou's lemma, we have $\|\tilde{f}\|_{\lambda,t}\leq 1$, and according to Lemma \ref{fjtof0}, we have $|\xi\cdot \tilde{f}|^2=K^{h}_{\xi,\Psi,\lambda}(t)$, thus $K^{h}_{\xi,\Psi,\lambda}(t)\leq\frac{|\xi\cdot \tilde{f}|^2}{\|\tilde{f}\|_{\lambda,t}^2}$. Note that $\|\tilde{f}\|_{\lambda,t}\leq 1$ implies $\tilde{f}\in A^2(M_T,h)$, which means $K^{h}_{\xi,\Psi,\lambda}(t)\geq\frac{|\xi\cdot \tilde{f}|^2}{\|\tilde{f}\|_{\lambda,t}^2}$. Then we get that $K^{h}_{\xi,\Psi,\lambda}(t)=\frac{|\xi\cdot \tilde{f}|^2}{\|\tilde{f}\|_{\lambda,t}^2}$.
\end{proof}

Recall that $Z_0$ is a subset of $M$, and $J_{z_0}$ is an $\mathcal{O}_{M,z_0}-$submodule of $J(E,\Psi)_{z_0}$ such that $I(h,\Psi)_{z_0}\subset J_{z_0}$ for any $z_0\in Z_0$. For any $t\geq T$, recall that
\[A^2(M_t,h)\cap J:=\left\{f\in A^2(M_t,h) :  f_{z_0}\in\mathcal{O}(K_M)_{z_0}\otimes J_{z_0}, \text{for\ any\ } z_0\in Z_0\right\}.\]
Following from Lemma \ref{Jclosed}, we know that $A^2(M_T,h)\cap J$ is a closed subspace of $A^2(M_T,h)$. Let $f\in A^2(M_T,h)$, such that $f\notin A^2(M_T,h)\cap J$. Recall the minimal $L^2$ integral (\cite{GMY5}) related to $J$ as follows:
\begin{flalign*}
	\begin{split}
		C(\Psi,h,J,f,M_T):=\inf\bigg\{\int_{M_T}|\tilde{f}|^2_h :  (\tilde{f}-f)_{z_0}\in (\mathcal{O}(K_M))_{z_0}\otimes J_{z_0}  \text{for\ any\ } z_0\in Z_0&\\
		\& \ \tilde{f}\in H^0(M_T,\mathcal{O}(K_M\otimes E))&\bigg\}.
	\end{split}
\end{flalign*}
Then the following lemma holds.

\begin{Lemma}\label{B=C}
	Assume that $C(\Psi,h,J,f,M_T)\in (0,+\infty)$, then
	\begin{equation}
		C(\Psi,h,J,f,M_T)=\sup_{\substack{\xi\in A^2(M_T,h)^*\setminus\{0\}\\ \xi|_{A^2(M_T,h)\cap J}\equiv 0}}\frac{|\xi\cdot f|^2}{K^{h}_{\xi,\Psi,\lambda}(T)}.
	\end{equation}
\end{Lemma}

\begin{proof}
	Denote that $(\tilde{f}-f)\in J$ if $(\tilde{f}-f)_{z_0}\in (\mathcal{O}(K_M))_{z_0}\otimes J_{z_0}$ for any $z_0\in Z_0$. Note that $\xi\cdot\tilde{f}=\xi\cdot f$ for any $\tilde{f}\in A^2(M_T,h)$ with $(\tilde{f}-f)\in J$, and $\xi\in A^2(M_T,h)^*$ satisfying $\xi|_{A^2(M_T,h)\cap J}\equiv 0$. Then we have
	\begin{flalign*}
		\begin{split}
			K^{h}_{\xi,\Psi,\lambda}(T)&=\sup_{g\in A^2(M_T,h)}\frac{|\xi\cdot g|^2}{\int_{M_T}|g|^2_h}\\
			&\geq\sup_{\substack{\tilde{f}\in A^2(M_T,h)\\ (\tilde{f}-f)\in J}}\frac{|\xi\cdot \tilde{f}|^2}{\int_{M_T}|\tilde{f}|^2_h}\\
			&= \sup_{\substack{\tilde{f}\in A^2(M_T,h)\\ (\tilde{f}-f)\in J}}\frac{|\xi\cdot f|^2}{\int_{M_T}|\tilde{f}|^2_h}.
		\end{split}
	\end{flalign*}
	Thus we get that
	\begin{flalign*}
		\begin{split}
			&\sup_{\substack{\xi\in A^2(M_T,h)^*\setminus\{0\}\\ \xi|_{A^2(M_T,h)\cap J}\equiv 0}}\frac{|\xi\cdot f|^2}{K^{h}_{\xi,\Psi,\lambda}(T)}\\
			\leq&\inf_{\substack{\tilde{f}\in A^2(M_T,h)\\ (\tilde{f}-f)\in J}}\int_{M_T}|\tilde{f}|^2_h\\
			=&C(\Psi,h,J,f,M_T).
		\end{split}
	\end{flalign*}
	
	Since $A^2(M_T,h)$ is a Hilbert space, and $A^2(M_T,h)\cap J$ is a closed proper subspace of $A^2(M_T,h)$, there exists a closed subspace $H$ of $A^2(M_T,h)$ such that $H=(A^2(M_T,h)\cap J)^{\bot}\neq\{0\}$. Then for $f\in A^2(M_T,h)$, we can make the decomposition $f=f_J+f_H$, such that $f_J\in A^2(M_T,h)\cap J$, and $f_H\in H$. Note that the linear functional $\xi_f$ defined as follows:
	\[\xi_f\cdot g:=\int_{M_T}\langle g,f_H\rangle_h, \ \forall g\in A^2(M_T,h),\]
	satisfies that $\xi_f\in A^2(M_T,h)^*\setminus\{0\}$ and $\xi_f|_{A^2(M_T,h)\cap J}\equiv 0$. Then we have
	\[\sup_{\substack{\xi\in A^2(M_T,h)^*\setminus\{0\}\\ \xi|_{A^2(M_T,h)\cap J}\equiv 0}}\frac{|\xi\cdot f|^2}{K^{h}_{\xi,\Psi,\lambda}(T)}\\
	\geq\frac{|\xi_f\cdot f|^2}{K^{h}_{\xi_f,\Psi,\lambda}(T)}.\]
	Besides, we can know that
	\[K^{h}_{\xi_f,\Psi,\lambda}(T)=\sup_{u\in A^2(M_T,h)}\frac{|\int_{M_T} \langle u,f_H\rangle_h|^2}{\int_{M_T}|u|^2_h}\leq\int_{M_T}|f_H|^2_h,\]
	and
	\[\xi_f\cdot f=\xi_f\cdot (f_J+f_H)=\xi_f\cdot f_H=\int_{M_T}|f_H|^2_h.\]
	Then we have
	\[\frac{|\xi_f\cdot f|^2}{K^{h}_{\xi_f,\Psi,\lambda}(T)}\geq\int_{M_T}|f_H|^2_h\geq C(\Psi,h,J,f,M_T),\]
	which implies that
	\[\sup_{\substack{\xi\in A^2(M_T,h)^*\setminus\{0\}\\ \xi|_{A^2(M_T,h)\cap J}\equiv 0}}\frac{|\xi\cdot f|^2}{K^{h}_{\xi,\Psi,\lambda}(T)}\\
	\geq C(\Psi,h,J,f,M_T).\]
	
	Lemma \ref{B=C} is proved.
\end{proof}

Let $\xi\in A^2(M_T,h)^*$, and recall that the Bergman kernel related to $\xi$ is
\[K^{h}_{\xi,\Psi,\lambda}(t):=\sup_{f\in A^2(M_T,h)}\frac{|\xi\cdot f|^2}{\|f\|^2_{\lambda,t}}\]
for any $t\in [T,+\infty)$ and $\lambda>0$. We state the following Lemma.

\begin{Lemma}\label{upper-semi}
	$K^{h}_{\xi,\Psi,\lambda}(t)$ is upper-semicontinuous with respect to $t\in[T,+\infty)$, i.e., for any sequence $\{t_j\}_{j=1}^{\infty}$ in $[T,+\infty)$ such that $\lim_{j\rightarrow+\infty}t_j=t_0\in [T,+\infty)$, we have
	\[\limsup_{j\rightarrow+\infty}K^{h}_{\xi,\Psi,\lambda}(t_j)\leq K^{h}_{\xi,\Psi,\lambda}(t_0).\]
\end{Lemma}

\begin{proof}
	
	Denote that
	\[K(t):=K^{h}_{\xi,\Psi,\lambda}(t)\]
	for any $t\in [T,+\infty)$. It can be seen that
	\[e^{\lambda(s-t)}\|f\|^2_{\lambda,s}\leq\|f\|_{\lambda,t}^2\leq\|f\|_{\lambda,s}^2\]
	for any $t>s\geq T$ and $f\in A^2(M_T,h)$. Note that $K(s)=0$ for some $s\geq T$ induces $K(t)=0$ for any $t\geq T$. Then it suffices to prove Lemma \ref{upper-semi} for $K(t_0)\in (0,+\infty)$ and $K(t_j)\in (0,+\infty)$, $\forall j\in\mathbb{N}_+$.
	
	We assume that $\{t_{k_j}\}$ is the subsequence of $\{t_j\}$ such that
	\[\lim_{j\rightarrow+\infty}K(t_{k_j})=\limsup_{j\rightarrow+\infty}K(t_j).\]
	By Lemma \ref{sup=max}, there exists a sequence of $E-$valued holomorphic $(n,0)$ forms $\{f_j\}$ on $M_T$ such that $f_j\in A^2(M_T,h)$, $\|f_j\|_{\lambda,t_j}=1$, and $|\xi\cdot f_j|^2=K(t_j)$, for any $j\in\mathbb{N}_+$. Since $\{t_j\}$ is bounded in $\mathbb{C}$, there exists some $s_0<+\infty$, such that $t_j<s_0$ for any $j$, which implies that
	\[\int_{M_T}|f_j|_h^2\leq e^{\lambda (s_0-T)}\|f_j\|^2_{\lambda,t_j}=e^{\lambda (s_0-T)}, \ \forall j\in\mathbb{N}_+.\]
	Then following from Montel's theorem, we can get a subsequence of $\{f_{k_j}\}$ (denoted by $\{f_{k_j}\}$ itself) compactly convergent to an $E-$valued holomorphic $(n,0)$ form $f_0$ on $M_T$. According to Fatou's lemma, we have
	\begin{flalign*}
		\begin{split}
			\|f_0\|_{\lambda,t_0}&=\int_{M_T}|f_0|_h^2e^{-\lambda\max\{\Psi+t_0,0\}}\\
			&=\int_{M_T}\lim_{j\rightarrow+\infty}|f_{k_j}|_h^2e^{-\lambda\max\{\Psi+t_{k_j},0\}}\\
			&\leq\liminf_{j\rightarrow+\infty}\int_{M_T}|f_{k_j}|_h^2e^{-\lambda\max\{\Psi+t_{k_j},0\}}\\
			&=\liminf_{j\rightarrow+\infty}\|f_{k_j}\|_{\lambda,t_j}=1.
		\end{split}
	\end{flalign*}
	Then $\int_{M_T}|f_0|_h^2\leq e^{\lambda (t_0-T)} \|f_0\|^2_{\lambda,t_0}\leq e^{\lambda (s_0-T)}<+\infty$, which implies that $f_0\in A^2(M_T,h)$. Lemma \ref{fjtof0} shows that $|\xi\cdot f_0|^2=\lim_{j\rightarrow+\infty}|\xi\cdot f_{k_j}|^2=\limsup_{j\rightarrow+\infty}K(t_j)$. Thus
	\[K(t_0)\geq\frac{|\xi\cdot f_0|^2}{\|f_0\|^2_{\lambda,t_0}}\geq\limsup_{j\rightarrow+\infty}K(t_j),\]
	which means that $K(t)$ is upper semi-continuous with respect to $t\in [T,+\infty)$.
\end{proof}

\section{Proof of Theorem \ref{concavity}}

We prove Theorem \ref{concavity} by using Lemma \ref{L2ext}.

\begin{proof}[Proof of Theorem \ref{concavity}]
	Denote that $\Omega:=M_T\times U_T$. Denote that $\pi_1,\pi_2$ are the natural projections from $\Omega$ to $M_T$ and $U_T$. Let $E':=\pi_1^*(M_T)$ be a vector bundle on $\Omega$. Let
	\[\tilde{\Psi}:=\lambda\max\{\Psi(z)+\text{Re\ }w,0\}\]
	for any $(z,w)\in \Omega$ with $z\in M_T$ and $w\in U_T$. Then $\tilde{\Psi}$ is a plurisubharmonic function on $\Omega_T:=M_T\times U_T$, where it can be seen that $\Omega$ is a Stein manifold.
	
	Denote that
	\[K(w):=K^{h}_{\xi,\Psi,\lambda}(\text{Re\ }w)\]
	for any $w\in U_T$. We prove that $\log K(w)$ is a subharmonic function with respect to $w\in U_T$.
	
	Firstly we prove that $\log K(w)$ is upper semicontinuous. Let $w_j\in U_T$ such that $\lim_{\lambda\rightarrow+\infty}w_j=w_0\in U_T$. Then $\lim_{j\rightarrow+\infty}\text{Re\ }w_j=\text{Re\ }w_0\in [T,+\infty)$. Following from Lemma \ref{upper-semi}, we get that
	\[\limsup_{j\rightarrow+\infty}\log K(w_j)\leq \log K(w_0).\]
	Thus $\log K(w)$ is upper semicontinuous with respect to $w\in U_T$.
	
	Secondly we prove that $\log K(w)$ satisfies the sub-mean value inequality on $U_T$.
	
	Let $w_0\in U_T$, and $\Delta(w_0,r)\subset U_T$ be the disc centered at $w_0$ with radius $r$. Let $\Omega':=M_T\times\Delta(w_0,r)\subset \Omega$ be a submanifold of $\Omega$. Let $f_0\in A^2(M_T,h)$ such that
	\[K(w_0)=\frac{|\xi\cdot f_0|^2}{\|f_0\|^2_{\lambda,\text{Re\ }w_0}}\]
	by Lemma \ref{sup=max}.
	
	Note that  $M_T$ is a Stein manifold, and $\tilde{\Psi}(z,w)=\Psi_{\lambda,\text{Re\ }w}=\lambda\max\{\Psi(z)+\text{Re\ }w,0\}$ is a bounded plurisubharmonic function on $\Omega'$. Using Lemma \ref{L2ext}, we can get an $E'-$valued holomorphic $(n+1,0)$ form $\tilde{f}$ on $\Omega'$ such that $\frac{\tilde{f}}{dw}|_{M_T\times\{w_0\}}=f_0$, and
	\begin{equation}\label{ineqL2ext}
		\frac{1}{\pi r^2}\int_{\Omega'}|\tilde{f}|_{\pi_1^*(h)}^2e^{-\tilde{\Psi}}\leq \int_{M_T}|f_0|_h^2e^{-\Psi_{\lambda,\text{Re\ }w_0}}.
	\end{equation}
	
	Denote that $\tilde{f}_w=\frac{\tilde{f}}{dw}|_{M_T\times\{w\}}$. Since the function $y=\log x$ is concave, according to Jensen's inequality and inequality (\ref{ineqL2ext}), we have
	\begin{flalign}\label{ineqJensen}
		\begin{split}
			\log\|f_0\|^2_{\lambda,\text{Re\ }w_0}&=\log\left(\int_{M_T}|f_0|_h^2e^{-\Psi_{\lambda,\text{Re\ }w_0}}\right)\\
			&\geq\log\left(\frac{1}{\pi r^2}\int_{\Omega'}|\tilde{f}|_h^2e^{-\tilde{\Psi}}\right)\\
			&=\log\left(\frac{1}{\pi r^2}\int_{\Delta(w_0,r)}\left(\int_{M_T\times\{w\}}|\tilde{f}_w|_h^2e^{-\Psi_{\lambda,\text{Re\ }w}}\right)d\mu_{\Delta_{w_0,r}}(w)\right)\\
			&\geq\frac{1}{\pi r^2}\int_{\Delta(w_0,r)}\log\left(\|\tilde{f}_w\|^2_{\lambda,\text{Re\ }w}\right)d\mu_{\Delta_{w_0,r}}(w)\\
			&\geq\frac{1}{\pi r^2}\int_{\Delta(w_0,r)}\left(\log|\xi\cdot \tilde{f}_w|^2-\log K(w)|\right)d\mu_{\Delta_{w_0,r}}(w).
		\end{split}
	\end{flalign}
	Where $\mu_{\Delta_{w_0,r}}$ is the Lebesgue measure on $\Delta_{w_0,r}$. It follows from Lemma \ref{xihol} that $\xi\cdot \tilde{f}_w$ is holomorphic with respect to $w$, which implies that $\log|\xi\cdot\tilde{f}_w|^2$ is subharmonic with respect to $w$. Then we have
	\[\log|\xi\cdot f_0|^2\leq\frac{1}{\pi r^2}\int_{\Delta(w_0,r)}\log|\xi\cdot \tilde{f}_w|^2d\mu_{\Delta_{w_0,r}}(w).\]
	Combining with inequlity (\ref{ineqJensen}), we get
	\[\log\|f_0\|^2_{\lambda,\text{Re\ }w_0}\geq\log|\xi\cdot f_0|^2-\frac{1}{\pi r^2}\int_{\Delta(w_0,r)}\log K(w)d\mu_{\Delta_{w_0,r}}(w),\]
	which means
	\[\log K(w_0)\leq\frac{1}{\pi r^2}\int_{\Delta(w_0,r)}\log K(w)d\mu_{\Delta_{w_0,r}}(w).\]
	
	Since $\log K(w)$ is upper semicontinuous and satisfies the sub-mean value inequality on $U_T$, we know that $\log K(w)$ is a subharmonic function on $U_T$.
\end{proof}

\section{Proof of Theorem \ref{increasing}}
In this section, we give the proof of Theorem \ref{increasing}. We need the following lemma.

\begin{Lemma}[see \cite{Demaillybook}]\label{Re}
	Let $D=I+\sqrt{-1}\mathbb{R}:=\{z=x+\sqrt{-1}y\in\mathbb{C} : x\in I, y\in\mathbb{R}\}$ be a subset of $\mathbb{C}$, where $I$ is an interval in $\mathbb{R}$. Let $\phi(z)$ be a subharmonic function on $D$ which is only dependent on $x=\text{Re\ }z$. Then $\phi(x):=\phi(x+\sqrt{-1}\mathbb{R})$ is a convex function with respect to $x\in I$.
\end{Lemma}

\begin{proof}[Proof of Theorem \ref{increasing}]
	It follows from Theorem \ref{concavity} that $\log K^{h}_{\xi,\Psi,\lambda}(\text{Re\ } w)$ is subharmonic with respect to $w\in (T,+\infty)+\sqrt{-1}\mathbb{R}$. Note that $\log K^{h}_{\xi,\Psi,\lambda}(\text{Re\ } w)$ is only dependent on $\text{Re\ }w$, then following from Lemma \ref{Re}, we get that $\log K^{h}_{\xi,\Psi,\lambda}(t)=\log K^{h}_{\xi,\Psi,\lambda}(t+\sqrt{-1}\mathbb{R})$ is convex with respect to $t\in (T,+\infty)$. Combining with Lemma \ref{upper-semi}, we get that $\log K^{h}_{\xi,\Psi,\lambda}(t)$ is convex with respect to $t\in [T,+\infty)$, which implies that $-\log K^{h}_{\xi,\Psi,\lambda}(t)+t$ is concave with respect to $t\in [T,+\infty)$. Then for any $\xi\in  A^2(M_T,h)^*$ with $\xi|_{ A^2(M_T,h)\cap J }\equiv 0$, to prove that $\log -K^{h}_{\xi,\Psi,\lambda}(t)+t$ is increasing, we only need to prove that $\log -K^{h}_{\xi,\Psi,\lambda}(t)+t$ has a lower bound on $[T,+\infty)$.
	
	Using Lemma \ref{sup=max}, we obtain that there exists $f_t\in  A^2(M_T,h)$ for any $t\in [T,+\infty)$, such that $\xi\cdot f_t=1$ and
	\begin{equation}\label{K=ft}
		K^{h}_{\xi,\Psi,\lambda}(t)=\frac{1}{\|f_t\|_{\lambda,t}^2}.
	\end{equation}
	In addition, according to Lemma \ref{L2mthod}, there exists an $E-$valued holomorphic $(n,0)$ form $\tilde{F}$ on $M_T$ such that
	\begin{equation}\label{tildeF}
		\int_{M_T}|\tilde{F}-(1-b_{t}(\Psi))f_tF^2|_{\tilde{h}}^2e^{-\tilde{\varphi}+v_{t}(\Psi)}\leq C\int_{M_T}\mathbb{I}_{\{-t-1<\Psi<-t\}}|f_tF|_{\tilde{h}}^2,
	\end{equation}
	where $\tilde{\varphi}=\max\{\psi+T,2\log|F|\}$, $\tilde{h}=he^{-\psi}$, and $C$ is a positive constant independent of $F$ and $t$. Then it follows from inequality (\ref{tildeF}) that
	\begin{flalign}\label{tildeF2}
		\begin{split}
			&\int_{M_T}|\tilde{F}-(1-b_{t}(\Psi))f_tF^2|_{\tilde{h}}^2e^{-\tilde{\varphi}+v_{t}(\Psi)}\\
			\leq&C\int_{M_T}\mathbb{I}_{\{-t-1<\Psi<-t\}}|f_tF|_{\tilde{h}}^2\\
			\leq&Ce^{t+1}\int_{\{\Psi<-t\}}|f_t|_{h}^2.
		\end{split}
	\end{flalign}
	Denote that $\tilde{F}_t:=\tilde{F}/F^2$ on $M_T$, then $\tilde{F}_t$ is an $E-$valued holomorphic $(n,0)$ form on $M_T$. Note that $\tilde{\varphi}=2\log|F|$ and $\Psi=\psi-2\log|F|$ on $M_T$. Then inequality (\ref{tildeF2}) implies that
	\begin{equation}\label{tildeFt}
		\int_{M_T}|\tilde{F}_t-(1-b_{t}(\Psi))f_t|_{h}^2e^{v_{t}(\Psi)-\Psi}\leq Ce^{t+1}\int_{\{\Psi<-t\}}|f_t|_{h}^2<+\infty.
	\end{equation}
	According to inequality (\ref{tildeFt}), we can get that $(\tilde{F}_t-f_t)_{z_0}\in \mathcal{O}(K_M)_{z_0}\otimes I(h,\Psi)_{z_0}\subset \mathcal{O}(K_M)_{z_0}\otimes J_{z_0}$, which means that $\xi\cdot \tilde{F}_t=\xi\cdot f_t=1$. Besides, since $v_t(\Psi)\geq \Psi$, we have
	\begin{flalign*}
		\begin{split}
			&\left(\int_{M_T}|\tilde{F}_t-(1-b_{t}(\Psi))f_t|_{h}^2e^{v_{t}(\Psi)-\Psi}\right)^{1/2}\\
			\geq&\left(\int_{M_T}|\tilde{F}_t-(1-b_{t}(\Psi))f_t|_{h}^2\right)^{1/2}\\
			\geq&\left(\int_{M_T}|\tilde{F}_t|_{h}^2\right)^{1/2}-\left(\int_{M_T}|(1-b_{t}(\Psi))f_t|_{h}^2\right)^{1/2}\\
			\geq&\left(\int_{M_T}|\tilde{F}_t|_{h}^2\right)^{1/2}-\left(\int_{\{\Psi<-t\}}|f_t|_{h}^2\right)^{1/2}.
		\end{split}
	\end{flalign*}
	Combining with inequality (\ref{tildeFt}), we have
	\begin{flalign*}
		\begin{split}
			&\int_{M_T}|\tilde{F}_t|_{h}^2\\
			\leq&2\int_{M_T}|\tilde{F}_t-(1-b_{t}(\Psi))f_t|_{h}^2e^{v_{t}(\Psi)-\Psi}+2\int_{\{\Psi<-t\}}|f_t|_{h}^2\\
			\leq&2(Ce^{t+1}+1)\int_{\{\Psi<-t\}}|f_t|_{h}^2.
		\end{split}
	\end{flalign*}
	Note that
	\begin{flalign*}
		\begin{split}
			\|f_t\|^2_{\lambda,t}=&\int_{M_T}|f_t|_{h}^2e^{-\Psi_{\lambda,t}}\\
			=&\int_{\{\Psi<-t\}}|f_t|_{h}^2+\int_{\{T>\Psi\geq -t\}}|f_t|_{h}^2e^{-\lambda(\Psi+t)}\\
			\geq&\int_{\{\Psi<-t\}}|f_t|_{h}^2.
		\end{split}
	\end{flalign*}
	Then we have
	\[\int_{M_T}|\tilde{F}_t|_{h}^2\leq 2(Ce^{t+1}+1)\|f_t\|_{\lambda,t}^2\leq C_1\frac{e^t}{K^{h}_{\xi,\Psi,\lambda}(t)},\]
	where $C_1$ is a positive constant independent on $t$. In addition, $\xi\cdot \tilde{F}_t=1$ implies that
	\[\int_{M_T}|\tilde{F}_t|_{h}^2=\|\tilde{F}_t\|^2_{\lambda,T}\geq (K^{h}_{\xi,\Psi,\lambda}(T))^{-1}.\]
	Then we get that
	\[-\log K^{h}_{\xi,\Psi,\lambda}(t)+t\geq C_2,\ \forall t\in [T,+\infty),\]
	where $C_2:=\log (C_1^{-1}K^{h}_{\xi,\Psi,\lambda}(T))$ is a finite constant. Since $-\log K^{h}_{\xi,\Psi,\lambda}(t)+t$ is concave, we get that $-\log K^{h}_{\xi,\Psi,\lambda}(t)+t$ is increasing with respect to $t\in [T,+\infty)$.
\end{proof}

\section{Proofs of Corollary \ref{L2integral} and Corollary \ref{SOPE}}
In this section, we give the proofs of Corollary \ref{L2integral} and Corollary \ref{SOPE}. Before the proofs, we do some preparations.

Let $h$ be a measurable metric on $E$ satisfying that $h$ has a positive locally lower bound. Let $(M,E,\Sigma, M_j, h,h_{j,s})$ be a singular metric on $E$. Assume that $\Theta_{h}\geq_{Nak}^s 0$. Let $f$ be a holomorphic $(n,0)$ form on $M_{t_0}=\{\Psi<-t_0\}$ for some $t_0\geq T$ such that $f\in A^2(M_{t_0},h)$. Let $z_0\in M$, and assume that $a_{z_0}^f(\Psi;h)<+\infty$. According to Remark \ref{a>0}, we know that $a_{z_0}^f(\Psi;h)\in (0,+\infty)$. 

Let $p>2a_{z_0}^f(\Psi;h)$ and $\lambda>0$. Let $\xi\in A^2(M_{t_0},h)^*\setminus\{0\}$ satisfying $\xi|_{A^2(M_{t_0},h)\cap J_p}\equiv 0$, where $J_p:=I(h,p\Psi)_{z_0}$. Denote that
\[K_{\xi,p,\lambda}(t):=\sup_{\tilde{f}\in A^2(M_{t_0},h)}\frac{|\xi\cdot \tilde{f}|^2}{\|\tilde{f}\|^2_{p,\lambda,t}},\]
where
\[\|\tilde{f}\|_{p,\lambda,t}:=\left(\int_{M_{t_0}}|\tilde{f}|_h^2e^{-\lambda\max\{p\Psi+t,0\}}\right)^{1/2},\]
and $t\in [pt_0,+\infty)$. Note that
\[p\Psi=\min\{p\psi+(2\lceil p\rceil-2p)\log|F|-2\log|F^{\lceil p\rceil}|,-pt_0\}\]
on $\{p\Psi<-pt_0\}$ for any $p>0$, where $\lceil p\rceil :=\min\{m\in\mathbb{Z} : m\geq p\}$. Then definition of $J_p$ shows that $f\in A^2(M_{t_0},h)\setminus(A^2(M_{t_0},h)\cap J_p)$, which implies that $A^2(M_{t_0},h)\cap J_p$ is a proper subspace of $A^2(M_{t_0},h)$, and $K_{\xi,p,\lambda}(pt_0)\in (0,+\infty)$. Then Theorem \ref{increasing} tells us that $-\log K_{\xi,p,\lambda}(t)+t$ is increasing with respect to $t\in [pt_0,+\infty)$, which implies that
\begin{equation}\label{-logK(t)+t}
	-\log K_{\xi,p,\lambda}(t)+t\geq -\log K_{\xi,p,\lambda}(pt_0)+pt_0, \ \forall t\in [pt_0,+\infty).
\end{equation}

Since $f\in A^2(M_{t_0},h)$, following from inequality (\ref{-logK(t)+t}), we get that
\[\|f\|^2_{p,\lambda,t}\geq\frac{|\xi\cdot f|^2}{K_{\xi,p,\lambda}(t)}\geq e^{-t+pt_0}\frac{|\xi\cdot f|^2}{K_{\xi,p,\lambda}(pt_0)}, \ \forall t\in [pt_0,+\infty).\]
In addition, since $f\notin A^2(M_{t_0},h)\cap J_p$, according to Lemma \ref{B=C}, we have
\begin{flalign}\label{f>e^-tC}
	\begin{split}
		\|f\|^2_{p,\lambda,t}&\geq\sup_{\substack{\xi\in A^2(M_{t_0},h)^*\setminus\{0\}\\ \xi|_{A^2(M_{t_0},h)\cap J_p}\equiv 0}}e^{-t+pt_0}\frac{|\xi\cdot f|^2}{K_{\xi,p,\lambda}(pt_0)}\\
		&=e^{-t+pt_0}C(p\Psi,h,J_p,f,M_{t_0}), \ \forall t\in [pt_0,+\infty).
	\end{split}
\end{flalign}
Note that for any $t\in [pt_0,+\infty)$,
\begin{equation}\label{f2}
	\|f\|^2_{p,\lambda,t}=\int_{\{p\Psi<-t\}}|f|_h^2+\int_{\{-pt_0>p\Psi\geq -t\}}|f|_h^2e^{-\lambda(p\Psi+t)}.
\end{equation}
Since for any $\lambda>0$,
\[\int_{\{-pt_0>p\Psi\geq -t\}}|f|_h^2e^{-\lambda(p\Psi+t)}\leq\int_{\{-pt_0>p\Psi\geq -t\}}|f|_h^2<+\infty,\]
and $\lim_{\lambda\rightarrow+\infty}e^{-\lambda(p\Psi+t)}=0$ on $\{-pt_0>p\Psi\geq -t\}$, according to Lebesgue's dominated convergence theorem, we have
\[\lim_{\lambda\rightarrow+\infty}\int_{\{-pt_0>p\Psi\geq -t\}}|f|_h^2e^{-\lambda(p\Psi+t)}=0.\]
Then equality (\ref{f2}) implies
\begin{equation}
	\lim_{\lambda\rightarrow+\infty}\|f\|^2_{p,\lambda,t}=\int_{\{p\Psi<-t\}}|f|_h^2, \ \forall t\in [pt_0,+\infty).
\end{equation}
Letting $\lambda\rightarrow+\infty$ in inequality (\ref{f>e^-tC}), we get that for any $t\in [pt_0,+\infty)$,
\begin{equation}\label{f>e^-tC2}
	\int_{\{p\Psi<-t\}}|f|_h^2\geq e^{-t+pt_0}C(p\Psi,h,J_p,f,M_{t_0}).
\end{equation}

Now we give the proof the Corollary \ref{L2integral}.
\begin{proof}[Proof of Corollary \ref{L2integral}]	
	Note that $J_p\subset I_+(h,2a_{z_0}^f(\Psi;h)\Psi)_{z_0}$ for any $p>2a_{z_0}^f(\Psi;h)$. Then we have
	\[C(p\Psi,h,J_p,f,M_{t_0})\geq C(\Psi,h,I_+(h,2a_{z_0}^f(\Psi;h)\Psi)_{z_0},f,M_{t_0}), \ \forall p>2a_{z_0}^f(\Psi;h).\]
	Since $\int_{M_{t_0}}|f|_h^2<+\infty$, it follows from Lebesgue's dominated convergence theorem and inequality (\ref{f>e^-tC2}) that
	\begin{flalign}
		\begin{split}
			&\int_{\{2a_{z_0}^f(\Psi;h)\Psi<-t\}}|f|_h^2\\
			=&\lim_{p\rightarrow 2a_{z_0}^f(\Psi;h)+0}\int_{\{p\Psi<-t\}}|f|_h^2\\
			\geq&\limsup_{p\rightarrow 2a_{z_0}^f(\Psi;h)+0}e^{-t+pt_0}C(p\Psi,h,J_p,f,M_{t_0})\\
			\geq&e^{-t+2a_{z_0}^f(\Psi;h)t_0}C(\Psi,h,I_+(h,2a_{z_0}^f(\Psi;h)\Psi)_{z_0},f,M_{t_0}),
		\end{split}
	\end{flalign}
	for any $t\in (2a_{z_0}^f(\Psi;h)t_0,+\infty)$. For $t=2a_{z_0}^f(\Psi;h)t_0$, it is clear that the above inequality also holds by the definition of $C(\Psi,h,I_+(h,2a_{z_0}^f(\Psi;h)\Psi)_{z_0},f,M_{t_0})$.
	
	Let $r:=e^{-t/2}$, and we get that Corollary \ref{L2integral} holds.
\end{proof}

In the following we give the proof of Corollary \ref{SOPE}.

\begin{proof}[Proof of Corollary \ref{SOPE}]
	For any $q>2a_{z_0}^f(\Psi;h)$, according to inequality (\ref{f>e^-tC2}), we get that for any $t\in [qt_0,+\infty)$,
	\begin{equation}\label{f>e^-tC2:2}
		\int_{\{q\Psi<-t\}}|f|_h^2\geq e^{-t+qt_0}C(q\Psi,h,J_q,f,M_{t_0}).
	\end{equation}
	
	It follows from Fubini's Theorem that
	\begin{flalign*}
		\begin{split}
			&\int_{\{\Psi<-t_0\}}|f|_h^2e^{-\Psi}\\
			=&\int_{\{\Psi<-t_0\}}\left(|f|_h^2\int_0^{e^{-\Psi}}\mathrm{d}s\right)\\
			=&\int_0^{+\infty}\left(\int_{\{\Psi<-t_0\}\cap\{s<e^{-\Psi}\}}|f|_h^2\right)\mathrm{d}s\\
			=&\int_{-\infty}^{+\infty}\left(\int_{\{q\Psi<-qt\}\cap\{\Psi<-t_0\}}|f|_h^2\right)e^t\mathrm{d}t.
		\end{split}
	\end{flalign*}
	Inequality (\ref{f>e^-tC2:2}) implies that for any $q>2a_o^f(\Psi;\varphi)$,
	\begin{flalign*}
		\begin{split}
			&\int_{t_0}^{+\infty}\left(\int_{\{q\Psi<-qt\}\cap\{\Psi<-t_0\}}|f|_h^2\right)e^t\mathrm{d}t\\
			\geq&\int_{t_0}^{+\infty}e^{-qt+qt_0}C(q\Psi,h,J_q,f,M_{t_0})\cdot e^t\mathrm{d}t\\
			=&\frac{1}{q-1}e^{t_0}C(q\Psi,h,J_q,f,M_{t_0}),
		\end{split}
	\end{flalign*}
	and
	\begin{flalign*}
		\begin{split}
			&\int_{-\infty}^{t_0}\left(\int_{\{q\Psi<-qt\}\cap\{\Psi<-t_0\}}|f|_h^2\right)e^t\mathrm{d}t\\
			\geq&\int_{-\infty}^{t_0}C(q\Psi,h,J_q,f,M_{t_0})
			\cdot e^t\mathrm{d}t\\
			=&e^{t_0}C(q\Psi,h,J_q,f,M_{t_0}).
		\end{split}
	\end{flalign*}
	Then we have
	\begin{equation}\label{q/q-1e^t_0}
		\int_{M_{t_0}}|f|_h^2e^{-\Psi}\geq \frac{q}{q-1}e^{t_0}C(q\Psi,h,J_q,f,M_{t_0}).
	\end{equation}
	for any $q>2a_{z_0}^f(\Psi;h)$. Note that $J_q\subset I_+(h,2a_{z_0}^f(\Psi;h)\Psi)_{z_0}$ for any $q>2a_{z_0}^f(\Psi;h)$, which implies
	\[C(q\Psi,h,J_q,f,M_{t_0})\geq C(\Psi,h,I_+(h,2a_{z_0}^f(\Psi;h)\Psi)_{z_0},f,M_{t_0}), \ \forall q>2a_{z_0}^f(\Psi;h).\]
	Then inequality (\ref{q/q-1e^t_0}) induces
	\begin{equation}\label{q/q-1}
		\int_{M_{t_0}}|f|_h^2e^{-\Psi}\geq \frac{q}{q-1}e^{t_0}C(\Psi,h,I_+(h,2a_{z_0}^f(\Psi;h)\Psi)_{z_0},f,M_{t_0}).
	\end{equation}
	Let $q\rightarrow 2a_{z_0}^f(\Psi;h)+0$, then inequality (\ref{q/q-1}) also holds for $q\geq 2a_{z_0}^f(\Psi;h)$. Thus if $q>1$ satisfying
	\begin{equation}
		\int_{M_{t_0}}|f|_h^2e^{-\Psi}< \frac{q}{q-1}e^{t_0}C(\Psi,h,I_+(h,2a_{z_0}^f(\Psi;h)\Psi)_{z_0},f,M_{t_0}),
	\end{equation}
	we have $q<2a_{z_0}^f(\Psi;h)$, which means that $f_{z_0}\in\mathcal{O}(K_M)_{z_0}\otimes I(h,q\Psi)_{z_0}$. Proof of Corollary \ref{SOPE} is done.
\end{proof}

\section{Appendix}
In this section, we give the proof of Lemma \ref{L2ext}. We firstly recall some notations and lemmas.

Let $M$ be a complex manifold. Let $\omega$ be a continuous hermitian metric on $M$. Let $dV_M$ be a continuous volume form on $M$. We denote by $L^2_{p,q}(M,\omega,dV_M)$ the spaces of $L^2$ integrable $(p,q)$ forms over $M$ with respect to $\omega$ and $dV_M$. It is known that $L^2_{p,q}(M,\omega,dV_M)$ is a Hilbert space.
\begin{Lemma}[see \cite{GMY5}]
\label{weakly convergence}
Let $\{u_n\}_{n=1}^{+\infty}$ be a sequence of $(p,q)$ forms in $L^2_{p,q}(M,\omega,dV_M)$ which is weakly convergent to $u$. Let $\{v_n\}_{n=1}^{+\infty}$ be a sequence of Lebesgue measurable real functions on $M$ which converges point-wisely to $v$. We assume that there exists a constant $C>0$ such that $|v_n|\le C$ for any $n$. Then $\{v_nu_n\}_{n=1}^{+\infty}$ weakly converges to $vu$ in $L^2_{p,q}(M,\omega,dV_M)$.
\end{Lemma}

\begin{Lemma}[see \cite{guan-zhou13ap}]\label{semipositive}
	Let $Q$ be a Hermitian vector bundle on a K\"{a}hler manifold $M$ of dimension $n$ with a K\"{a}hler metric $\omega$. Let $\theta$ be a continuous $(1,0)$ form on $M$. Then we have
	\[[\sqrt{-1}\theta\wedge\bar{\theta},\Lambda_{\omega}]\alpha=\bar{\theta}\wedge(\alpha\llcorner(\bar{\theta})^{\sharp}),\]
	for any $(n,1)$ form $\alpha$ with value in $Q$. Moreover, for any positive $(1,1)$ form $\beta$, we have $[\beta,\Lambda_{\omega}]$ is semipositive.
\end{Lemma}

Let $X$ be an $n-$dimensional complex manifold and $\omega$ be a hermitian metric on $X$. Let $Q$ be a vector bundle on $X$ with rank $r$. Let $D'':L^2(M,\wedge^{n,q}T^*M\otimes Q)\rightarrow L^2(M,\wedge^{n,q+1}T^*M\otimes Q)$ be the extension of $\bar{\partial}-$operator in the sense of distribution. Let $\{h_i\}_{i=1}^{+\infty}$ be a family of $C^2$ smooth hermitian metric on $Q$ and $h$ be a measurable metric on $Q$ such that $\lim_{i\to+\infty}h_i=h$ almost everywhere on $X$.  We assume that $\{h_i\}_{i=1}^{+\infty}$ and $h$ satisfy one of the following conditions,\\
$(A)$ $h_i$ is increasingly convergent to $h$ as $i\to+\infty$;\\
$(B)$ there exists a continuous metric $\hat{h}$ on $Q$ and a constant $C>0$ such that for any $i\ge 0$, $\frac{1}{C}\hat{h}\le h_i\le C\hat{h}$ and $\frac{1}{C}\hat{h}\le h\le C\hat{h}$.

Denote $\mathcal{H}_i:=L^2(X,K_X\otimes Q,h_i,dV_{\omega})$ and $\mathcal{H}:=L^2(X,K_X\otimes Q,h,dV_{\omega})$. Note that $\mathcal{H}\subset \mathcal{H}_i\subset \mathcal{H}_1$ for any $i\in\mathbb{Z}_{>0}$.

Denote $P_i:=\mathcal{H}_i\to \text{Ker}D''$ and $P:=\mathcal{H}\to \text{Ker}D''$ be the orthogonal projections with respect to $h_i$ and $h$ respectively.
\begin{Lemma}[\cite{GMY5}]\label{weakly converge lemma}
For any sequence of $Q$-valued $(n,0)$-forms $\{f_i\}_{i=1}^{+\infty}$ which satisfies $f_i\in\mathcal{H}_i$ and $||f_i||_{h_i}\le C_1$ for some constant $C_1>0$, there exists a $Q$-valued $(n,0)$-form $f_0\in \mathcal{H}$ such that there exists a subsequence of $\{f_i\}_{i=1}^{+\infty}$ (also denoted by $\{f_i\}_{i=1}^{+\infty}$) weakly converges to $f_0$ in $\mathcal{H}_1$ and $P_i(f_i)$ weakly converges to $P(f_0)$ in $\mathcal{H}_1$.
\end{Lemma}

We need the following result in Hilbert spaces.
\begin{Lemma}\label{weaklycon_twonorms}
    Let $\langle\cdot,\cdot\rangle_1$, $\langle\cdot,\cdot\rangle_2$ be two inner products on a vector space $H$ such that both $(H,\langle\cdot,\cdot\rangle_1)$ and $(H,\langle\cdot,\cdot\rangle_2)$ are Hilbert spaces. Assume that there exists some $C>0$ such that $\|\cdot\|_2\le C\|\cdot\|_1$, where $\|\cdot\|_1$, $\|\cdot\|_2$ are the norms induced by $\langle\cdot,\cdot\rangle_1$ and $\langle\cdot,\cdot\rangle_2$ respectively. Then for any sequence $\{x_j\}\subset H$ weakly convergent to $x\in H$ in $(H,\langle\cdot,\cdot\rangle_1)$, then $\{x_j\}$ also weakly converges to $x$ in $(H,\langle\cdot,\cdot\rangle_2)$. 
\end{Lemma}

\begin{proof}
For any $w\in H$, we denote a functional $L_w$ over $(H,\langle\cdot,\cdot\rangle_1)$ as follows:
\begin{flalign*}
		\begin{split}
			L_w \ : \ H&\longrightarrow\mathbb{C}\\
			z&\longmapsto \langle z,w\rangle_2.
		\end{split}
	\end{flalign*}
It is clear that $L_w$ is linear. In addition, for any $z\in H$, we have
\[|L_w(z)|=|\langle z,y\rangle_2|\le \|z\|_2\|w\|_2\le C\|w\|_2\|z\|_1.\]
Then $L_w$ is a continuous functional over $(H,\langle\cdot,\cdot\rangle_1)$, which implies that there exists some $Tw\in H$ such that
\[\langle z,w\rangle_2=L_w(z)=\langle z,Tw\rangle_1\]
for any $z\in H$ by Riesz representation theorem. We have that $T : H\to H$ is a continuous linear operator. It follows that any weakly convergent sequence in $(H,\langle\cdot,\cdot\rangle_1)$ is also a weakly convergent sequence in $(H,\langle\cdot,\cdot\rangle_2)$.
\end{proof}

\begin{Lemma}[see \cite{GMY5}]\label{B+lambdaI}
	Let $M$ be a complex manifold admitting a complete K\"{a}hler metric, and $\omega$ is a K\"{a}hler metric on $M$ (not necessarily complete). Let $(Q,h)$ be a hermitian vector bundle over $M$. Assume that $\eta$ and $g$ are smooth bounded positive functions on $M$ such that $\eta+g^{-1}$ are smooth bounded positive functions on $M$ such that $\eta+g^{-1}$ is a smooth bounded positive functions on $M$ and let $B:=[\eta\sqrt{-1}\Theta_Q-\sqrt{-1}\partial\bar{\partial}\eta-\sqrt{-1}g\partial\eta\wedge\bar{\partial}\eta,\Lambda_{\omega}]$. Assume that $\tilde{\lambda}\geq 0$ is a bounded continuous function on $M$ such that $B+\tilde{\lambda}I$ is positive definite everywhere on $\wedge^{n,q}T^*M\otimes Q$ for some $q\geq 1$. Then given a form $v\in L^2(M,\wedge^{n,q}T^*M\otimes Q)$ such that $D''v=0$ and $\int_M\langle (B+\tilde{\lambda I})^{-1}v,v\rangle_{Q,\omega}dV_{\omega}<+\infty$, there exists an approximate solution $u\in L^2(M,\wedge^{n,q-1}T^*M\otimes Q)$ and a correcting term $\tau\in L^2(M,\wedge^{n,q}T^*M\otimes Q)$ such that $D''u+P_h(\sqrt{\tilde{\lambda}}\tau)=v$, where $P_h:L^2(M,\wedge^{n,q}T^*M\otimes Q)\rightarrow \text{Ker\ } D''$ is the orthogonal projection and
	\begin{equation}
		\int_M(\eta+g^{-1})^{-1}|u|_{Q,\omega}^2dV_{\omega}+\int_M|\tau|_{Q,\omega}^2dV_{\omega}\leq \int_M\langle (B+\tilde{\lambda}I)^{-1}v,v\rangle_{Q,\omega}dV_{\omega}.
	\end{equation}
\end{Lemma}

\begin{Lemma}[see \cite{FN80}]\label{Regularization}
	Let $X$ be a Stein manifold and $\varphi$ a plurisubharmonic function on $X$. Then there exists a sequence
		$\{\varphi_{n}\}_{n=1,\cdots}$ of smooth strongly plurisubharmonic functions such that
		$\varphi_{n} \downarrow \varphi$. 
\end{Lemma}

\begin{Lemma}[Lemma 6.9 in \cite{Demailly82}]\label{barpartialv=h}
	Let $\Omega$ be an open subset of $\mathbb{C}^n$ and $Z$ be a complex analytic subset of $\Omega$. Assume that $v$ is a $(p,q-1)-$form with $L^2_{\text{loc}}$ coefficients and $h$ is an $L^1_{\text{loc}}$ $(p,q)-$form coefficients such that $\bar{\partial}v=h$ on $\Omega\setminus Z$ (in the sense of distribution theory). Then $\bar{\partial}v=h$ on $\Omega$.
\end{Lemma}

In the following, we give the proof of Lemma \ref{L2ext}.

Note that $M$ is a Stein manifold, there exists a smooth plurisubharmonic
exhaustion function $P$ on $M$. Let $M_j:=\{P<j\}$ $(k=1,2,...,) $. We choose $P$ such that
$M_1\ne \emptyset$.

Then $M_1 \Subset  M_2\Subset  ...\Subset
M_j\Subset  M_{j+1}\Subset  ...$ and $\cup_{j=1}^{+\infty} M_j=M$. Each $M_j$ is a Stein manifold.

For any smooth metric $\hat{h}$ on $M$, since $h$ has a positive locally lower bound, we can find some $C_K>0$ such that $|e|_h\geq C_K|e|_{\hat{h}}$ on $K$ for any compact subset $K$ of $M$ and any local holomorphic section $e$ of $E$. Then it follows from $\int_{M}|u|_h^2e^{-\tilde{\Psi}_{w_0}}<+\infty$ that $\int_K|u|^2_{\hat{h}}<+\infty$ for any compact subset $K$ of $M$.

\

\emph{Step 1: Regularization of $\tilde{\Psi}$. }

According to Lemma \ref{Regularization}, we can find a sequence of smooth strongly plurisubharmonic functions $\{\tilde{\Psi}_m\}_{m=1}^{+\infty}$ on $\Omega$ such that $\tilde{\Psi}_{m} \downarrow \tilde{\Psi}$ on $\Omega$.

Additionally, Let $r_i\in (0,r)$ be a sequence of real numbers such that $r_i\to r$ as $i\to+\infty$, and $D_i:=\{|w-w_0|<r_i\}\subset D$, $\Omega_i:=M\times D_i$. If for any $i$, there exists some extension $\tilde{u}_i$ of $u$ such that
\[\int_{\Omega_i}|\tilde{u}_i|_{p_2^*(h)}^2e^{-\tilde{\Psi}}\leq\frac{1}{\pi r_i^2}\int_{M}|u|_h^2e^{-\tilde{\Psi}_{w_0}},\]
then by Montel's theorem and the diagonal method, we can find an extension $\tilde{u}$ of $u$ on $\Omega$ such that
\[\int_{\Omega}|\tilde{u}|_{p_2^*(h)}^2e^{-\tilde{\Psi}}\leq\frac{1}{\pi r^2}\int_{M}|u|_h^2e^{-\tilde{\Psi}_{w_0}}.\]

Since $\tilde{\Psi}$ is bounded, and $M_j\times D_i$ is relatively compact in $\Omega$, combining with the above discussion, we can assume $\tilde{\Psi}_m$ is uniformly bounded in $M_j\times D$ with respect to $m$ for any fixed $j$ (see \cite{guan-zhou13ap}).

\

\emph{Step 2: Recall some constructions. }

Let $t_0\in (0,+\infty)$, $B>0$. In the following, to simplify our notations, we denote $b_{t_0,B}(t)$ by $b(t)$ and $v_{t_0,B}(t)$ by $v(t)$.

Let $\epsilon \in (0,\frac{1}{8}B)$. Let $\{v_\epsilon\}_{\epsilon \in
(0,\frac{1}{8}B)}$ be a family of smooth increasing convex functions on $\mathbb{R}$, such
that:
\par
(1) $v_{\epsilon}(t)=t$ for $t\ge-t_0-\epsilon$, $v_{\epsilon}(t)=constant$ for
$t<-t_0-B+\epsilon$;\par
(2) $v_{\epsilon}{''}(t)$ are convergence pointwisely
to $\frac{1}{B}\mathbb{I}_{(-t_0-B,-t_0)}$,when $\epsilon \to 0$, and $0\le
v_{\epsilon}{''}(t) \le \frac{2}{B}\mathbb{I}_{(-t_0-B+\epsilon,-t_0-\epsilon)}$
for ant $t \in \mathbb{R}$;\par
(3) $v_{\epsilon}{'}(t)$ are convergence pointwisely to $b(t)$ which is a continuous
function on $\mathbb{R}$ when $\epsilon \to 0$ and $0 \le v_{\epsilon}{'}(t) \le 1$ for any
$t\in \mathbb{R}$.\par
One can construct the family $\{v_\epsilon\}_{\epsilon \in (0,\frac{1}{8}B)}$  by
 setting
\begin{equation}\nonumber
\begin{split}
v_\epsilon(t):=&\int_{-\infty}^{t}(\int_{-\infty}^{t_1}(\frac{1}{B-4\epsilon}
\mathbb{I}_{(-t_0-B+2\epsilon,-t_0-2\epsilon)}*\rho_{\frac{1}{4}\epsilon})(s)ds)dt_1\\
&-\int_{-\infty}^{-t_0}(\int_{-\infty}^{t_1}(\frac{1}{B-4\epsilon}
\mathbb{I}_{(-t_0-B+2\epsilon,-t_0-2\epsilon)}*\rho_{\frac{1}{4}\epsilon})(s)ds)dt_1-t_0,
\end{split}
\end{equation}
where $\rho_{\frac{1}{4}\epsilon}$ is the kernel of convolution satisfying
$\text{supp}(\rho_{\frac{1}{4}\epsilon})\subset
(-\frac{1}{4}\epsilon,{\frac{1}{4}\epsilon})$.
Then it follows that
\begin{equation}\nonumber
v_\epsilon{''}(t)=\frac{1}{B-4\epsilon}
\mathbb{I}_{(-t_0-B+2\epsilon,-t_0-2\epsilon)}*\rho_{\frac{1}{4}\epsilon}(t),
\end{equation}
and
\begin{equation}\nonumber
v_\epsilon{'}(t)=\int_{-\infty}^{t}(\frac{1}{B-4\epsilon}
\mathbb{I}_{(-t_0-B+2\epsilon,-t_0-2\epsilon)}*\rho_{\frac{1}{4}\epsilon})(s)ds.
\end{equation}

Let $\psi_0:=p_1^*(2\log|w-w_0|-2\log r)$ be a plurisubharmonic function on $\Omega$. Let $\eta=s(-v_\epsilon(\psi_0))$ and $\phi=u(-v_\epsilon(\psi_0))$, where $s \in
C^{\infty}((0,+\infty))$ satisfies $s>0$ and $u\in C^{\infty}((0,+\infty))$, such that $s'(t)\neq 0$ for any $t$, $u''s-s''>0$
and $s'-u's=1$.

Recall that $(M,E,\Sigma,M_j,h,h_{j,m'})$ is a singular hermitian metric on $E$. Then there exists a sequence of hermitian metrics $\{h_{j,m'}\}_{m'=1}^{+\infty}$ on $M_{j+1}$ of class $C^2$ such that $\lim\limits_{m'\to+\infty}h_{j,m'}=h$ almost everywhere on $M_{j+1}$ and $\{h_{j,m'}\}_{m'=1}^{+\infty}$ satisfies the conditions of Definition \ref{singular nak}. We will fix $j$ until the last step (Step 9), thus we simply denote $h_{j,m'}$ by $h_{m'}$. Denote that $\tilde{h}:=p_2^*(h_{m'})e^{-\Phi_m}$, where $\Phi_m:=\tilde{\Psi}_m+\phi+\psi_0$.

\

\emph{Step 3: Solving $\bar{\partial}$-equation with error term. }

Set $B=[\eta \sqrt{-1}\Theta_{\tilde{h}}-\sqrt{-1}\partial \bar{\partial}
\eta\otimes\text{Id}_{E'}-\sqrt{-1}g\partial\eta \wedge\bar{\partial}\eta\otimes\text{Id}_{E'}, \Lambda_{\omega\wedge\omega_0}]$, where
$\omega_0=\frac{\sqrt{-1}}{2}dw\wedge d\bar{w}$ is the standard K\"{a}hler form on $\mathbb{C}$, and $g$ is a positive function. We will determine $g$ by calculations. On $M_j\times (D\setminus\{w_0\})$, direct calculation shows that
\begin{equation}\nonumber
	\begin{split}
		\partial\bar{\partial}\eta=&
		-s'(-v_{\epsilon}(\psi_0))\partial\bar{\partial}(v_{\epsilon}(\psi_0))
		+s''(-v_{\epsilon}(\psi_0))\partial(v_{\epsilon}(\psi_0))\wedge
		\bar{\partial}(v_{\epsilon}(\psi_0)),\\
		\partial\bar{\partial}\phi=&-u'(-v_{\epsilon}(\psi_0))\partial\bar{\partial}(v_{\epsilon}(\psi_0))+u''(-v_{\epsilon}(\psi_0))\partial(v_{\epsilon}(\psi_0))\wedge\bar{\partial}(v_{\epsilon}(\psi_0)),\\
		\eta\Theta_{\tilde{h}}=&\eta\partial\bar{\partial}\phi\otimes\text{Id}_{E'}+\eta\Theta_{h_{m'}\boxtimes h_0}+\eta\partial\bar{\partial}(\tilde{\Psi}_m)\otimes\text{Id}_{E'}+\eta\partial\bar{\partial}\psi_0\otimes \text{Id}_{E'}\\
		=&su''(-v_{\epsilon}(\psi_0))\partial(v_{\epsilon}(\psi_0))\wedge\bar{\partial}(v_{\epsilon}(\psi_0))\otimes\text{Id}_{E'}-su'(-v_{\epsilon}(\psi_0))\partial\bar{\partial}(v_{\epsilon}(\psi_0))\otimes\text{Id}_{E'}\\
		&+s\Theta_{p_2^*(h_{m'})}+s\partial\bar{\partial}(\tilde{\Psi}_m)\otimes\text{Id}_{E'}.
\end{split}
\end{equation}
Therefore,
\begin{equation}\nonumber
	\begin{split}
		&\eta\sqrt{-1}\Theta_{\tilde{h}}-\sqrt{-1}\partial\bar{\partial}\eta\otimes\text{Id}_E-\sqrt{-1}g\partial\eta\wedge\bar{\partial}\eta\otimes\text{Id}_E\\
		=&s\Theta_{p_2^*(h_{m'})}+s\partial\bar{\partial}(\tilde{\Psi}_m)\otimes\text{Id}_{E'}\\
		&+(s'-su')(v'_{\epsilon}(\psi_0)\sqrt{-1}\partial\bar{\partial}\psi_0+v''_{\epsilon}(\psi_0)\sqrt{-1}\partial\psi_0\wedge\bar{\partial}\psi_0)\otimes\text{Id}_{E'}\\
		&+((u''s-s'')-gs'^2)\sqrt{-1}\partial(v_{\epsilon}(\psi_0))\wedge\bar{\partial}(v_{\epsilon}(\psi_0))\otimes\text{Id}_{E'}.
	\end{split}
\end{equation}
We omit the composition item $(-v_{\epsilon}(\psi_0))$ after $s'-su'$ and $(u''s-s'')-gs'^2$ in the above equalities.

Note that $u''s-s''>0$. Let $g=\frac{u''s-s''}{s'^2}(-v_{\epsilon}(\psi_0))$. We have $\eta+g^{-1}=(s+\frac{s'^2}{u''s-s''})(-v_{\epsilon}(\psi_0))$. Note that $s'-su'=1$, $0\leq v'_{\epsilon}(\psi_0)\leq 1$. Then
\begin{equation*}
	\begin{split}
		&\eta\sqrt{-1}\Theta_{\tilde{h}}-\sqrt{-1}\partial\bar{\partial}\eta\otimes\text{Id}_{E'}-\sqrt{-1}\partial\eta\wedge\bar{\partial}\eta\otimes\text{Id}_{E'}\\
		=&s\Theta_{p_2^*(h_{m'})}+s\partial\bar{\partial}(\tilde{\Psi}_m)\otimes\text{Id}_{E'}\\
		&+v'_{\epsilon}(\psi_0)\sqrt{-1}\partial\bar{\partial}\psi_0\otimes\text{Id}_{E'}+v''_{\epsilon}(\psi_0)\sqrt{-1}\partial\psi_0\wedge\bar{\partial}\psi_0\otimes\text{Id}_{E'}\\
		=&v'_{\epsilon}(\psi_0)\sqrt{-1}\partial\bar{\partial}\psi_0\otimes\text{Id}_{E'}+v''_{\epsilon}(\psi_0)\sqrt{-1}\partial\psi_0\wedge\bar{\partial}\psi_0\otimes\text{Id}_{E'}\\
		&+s(\Theta_{h_{m'}}+\tilde{\lambda}_{m'}\omega\otimes\text{Id}_E)\wedge\omega_0\otimes\text{Id}_{D\times\mathbb{C}}+s\partial\bar{\partial}\tilde{\Psi}_m\otimes\text{Id}_{E'}\\
		&-\tilde{\lambda}_{m'}\omega\wedge\omega_0\otimes\text{Id}_{E'}\\
		\geq&v''_{\epsilon}(\psi_0)\sqrt{-1}\partial\psi_0\wedge\bar{\partial}\psi_0\otimes\text{Id}_{E'}-s\tilde{\lambda}_{m'}\omega\wedge\omega_0\otimes\text{Id}_{E'}.
	\end{split}
\end{equation*}
Here from Definition \ref{singular nak}, $\tilde{\lambda}_{m'}$ satisfies $\Theta_{h_{m'}}(E)\geq_{Nak}-\tilde{\lambda}_{m'}\omega\otimes\text{Id}_{E}$ on $M_j$.

It can be seen that $s(-v_{\epsilon}(\psi_0))$ is uniformly upper bounded on $M_j\times D$ with respect to $j,m,m',\epsilon$. Let $N_1$ be the uniformly upper bound of $s(-v_{\epsilon}(\psi_0))$ on $M_j\times D$. Then on $M_j\times (D\setminus\{w_0\})$, we have
\begin{equation*}
	\begin{split}
		&\eta\sqrt{-1}\Theta_{\tilde{h}}-\sqrt{-1}\partial\bar{\partial}\eta\otimes\text{Id}_{E'}-\sqrt{-1}\partial\eta\wedge\bar{\partial}\eta\otimes\text{Id}_{E'}\\
		\geq&v''_{\epsilon}(\psi_0)\sqrt{-1}\partial\psi_0\wedge\bar{\partial}\psi_0\otimes\text{Id}_{E'}-N_1\tilde{\lambda}_{m'}\omega\wedge\omega_0\otimes\text{Id}_{E'}.
	\end{split}
\end{equation*}
Then for any $E'-$valued $(n+1,1)$ form $\alpha$, we have
\begin{equation}\label{B+lambdaIdE'}
	\begin{split}
		&\langle(B+N_1\tilde{\lambda}_{m'}\text{Id}_{E'})\alpha,\alpha\rangle_{\tilde{h}}\\
		\geq&\langle[v''_{\epsilon}(\psi_0)\partial(\psi_0)\wedge\bar{\partial}(\psi_0)\otimes\text{Id}_{E'},\Lambda_{\omega\wedge\omega_0}]\alpha,\alpha\rangle_{\tilde{h}}\\
		=&\langle(v''_\epsilon(\psi_0)\bar{\partial}(\psi_0)
		\wedge(\alpha\llcorner(\bar{\partial}\psi_0)^{\sharp}))\alpha,\alpha\rangle_{\tilde
			h}.
	\end{split}
\end{equation}
It follows from Lemma \ref{semipositive} that $B+N_1\tilde{\lambda}_{m'}\text{Id}_{E'}$ is semipositive. Denote $\hat{\lambda}_{m'}:=\tilde{\lambda}_{m'}+\frac{1}{m'}$, then $\tilde{B}:=B+N_1\hat{\lambda}_{m'}\text{Id}_{E'}$ is positive. According to inequality (\ref{B+lambdaIdE'}), we have
\begin{equation}
	\label{cs inequality}
	\begin{split}
		|\langle
		v''_\epsilon(\psi_0)\bar{\partial}\psi_0\wedge\gamma,\tilde{\alpha}\rangle_
		{\tilde h}|^2=
		&|\langle
		v''_\epsilon(\psi_0)\gamma,\tilde{\alpha}\llcorner(\bar{\partial}\psi_0)^{\sharp}
		\rangle_{\tilde h}|^2\\
		\le&\langle
		(v''_\epsilon(\psi_0)\gamma,\gamma)
		\rangle_{\tilde h}
		(v''_\epsilon(\psi_0))|\tilde{\alpha}\llcorner(\bar{\partial}\psi_0)^{\sharp}|^2_{\tilde
			h}\\
		=&\langle
		(v''_\epsilon(\psi_0)\gamma,\gamma)
		\rangle_{\tilde h}
		\langle
		(v''_\epsilon(\psi_0))\bar{\partial}\psi_0\wedge
		(\tilde{\alpha}\llcorner(\bar{\partial}\psi_0)^{\sharp}),\tilde{\alpha}
		\rangle_{\tilde h}\\
		\le&\langle
		(v''_\epsilon(\psi_0)\gamma,\gamma)
		\rangle_{\tilde h}
		\langle
		\tilde{B}\tilde{\alpha},\tilde{\alpha})
		\rangle_{\tilde h}
	\end{split}
\end{equation}
for any $E'-$valued $(n+1,0)$ form $\gamma$ and $E'-$valued $(n+1,1)$ form $\tilde{\alpha}$.

Let $f:=u\wedge dw$ be the trivial extension of $u$ from $M_j\times\{w_0\}$ to $\Omega$.
Then $\mu:=\bar{\partial}\big((1-v_{\epsilon}'(\psi_0))f\big)$ is well defined and smooth on $M_j\times D$. Note that
\[\mu=-\bar{\partial}v'_{\epsilon}(\psi_0)\wedge f.\]
Take $\gamma=f$, $\tilde{\alpha}=\tilde{B}^{-1}\mu$. Then it follows from inequality (\ref{cs inequality}) that
\[\langle\tilde{B}^{-1}\mu,\mu\rangle_{\tilde{h}}\leq v''_{\epsilon}(\psi_0)|f|^2_{\tilde{h}}.\]
Thus we have
\begin{equation}
    \int_{M_j\times (D\setminus\{w_0\})}\langle
\tilde{B}^{-1}\mu,\mu\rangle_{\tilde h}
\leq\int_{M_j\times (D\setminus\{w_0\})}v''_\epsilon(\psi_0)|f|^2_{\tilde{h}}
\end{equation}

Recall that $\tilde{h}=p_2^*(h_{m'})e^{-\Phi_m}$ and $\Phi_m=\phi+\tilde{\Psi}_m+\psi_0$. Note that $0\le
v_{\epsilon}''(t) \le \frac{2}{B}\mathbb{I}_{(-t_0-B+\epsilon,-t_0-\epsilon)}$, $e^{-\phi}$ is bounded function on $M_j\times D$, $h_{m'}\leq h$, and $\tilde{\Psi}_m$ is lower bounded on $\Omega$. Then
    \begin{equation*}
        \begin{split}
        &\int_{M_j\times (D\setminus\{w_0\})}v''_\epsilon(\psi_0)|f|^2_{\tilde{h}}\\
        \le & e^{t_0+B-\epsilon} \sup_{M_j\times D}(e^{-\phi-\tilde{\Psi}_m})\int_{M_j\times D}\frac{2}{B}\mathbb{I}_{(-t_0-B+\epsilon,-t_0-\epsilon)}|f|_{h\boxtimes h_0}^2<+\infty.
    \end{split}
    \end{equation*}

It is clear that $M_j\times (D\setminus\{w_0\})$ carries a complete K\"ahler metric since $M_j$ is Stein. Then it follows from Lemma \ref{B+lambdaI} that there exists
\[u_{m,m',\epsilon,j}\in L^2(M_j\times (D\setminus\{w_0\}), K_{\Omega}\otimes {E'},p_2^*(h_{m'})e^{-\Phi_m}),\]
\[\mathbf{h}_{m,m',\epsilon,j}\in L^2(M_j\times (D\setminus\{w_0\}), \wedge^{n+1,1}T^*\Omega \otimes {E'},p_2^*(h_{m'})e^{-\Phi_m}),\]
such that $\bar\partial u_{m,m',\epsilon,j}+P_{m,m'}\big(\sqrt{N_1 \hat{\lambda}_{m'}}\mathbf{h}_{m,m',\epsilon,j}\big)=\mu$ holds on $M_j\times (D\setminus\{w_0\})$ where $P_{m,m'}:L^2(M_j\times (D\setminus\{w_0\}), \wedge^{n+1,1}T^*\Omega \otimes {E'},p_2^*(h_{m'})e^{-\Phi_m})\to \text{Ker}{D''}$ is the orthogonal projection, and

\begin{equation}\nonumber
\begin{split}
&\int_{M_j\times (D\setminus\{w_0\})}
\frac{1}{\eta+g^{-1}}|u_{m,m',\epsilon,j}|^2_{p_2^*(h_{m'})}e^{-\Phi_m}+
\int_{M_j\times (D\setminus\{w_0\})}
|\mathbf{h}_{m,m',\epsilon,j}|^2_{p_2^*(h_{m'})}e^{-\Phi_m} \\
\le&
\int_{M_j\times (D\setminus\{w_0\})}
\langle
(B+N_1 \hat{\lambda}_{m'} \text{Id}_{E'})^{-1}\mu,\mu
\rangle_{\tilde h}\\
\le&
\int_{M_j\times (D\setminus\{w_0\})}
v''_{\epsilon}(\psi_0)|f|^2_{p_2^*(h_{m'})}e^{-\Phi_m}\\
<&+\infty.
\end{split}
\end{equation}

Assume that we can choose
$\eta$ and $\phi$ such that
$(\eta+g^{-1})^{-1}=e^{v_\epsilon(\psi_0)}e^{\phi}$. Then we have

\begin{equation}
\label{estimate 1}
\begin{split}
&\int_{M_j\times (D\setminus\{w_0\})}
|u_{m,m',\epsilon,j}|^2_{p_2^*(h_{m'})} e^{v_\epsilon(\psi_0)-\tilde{\Psi}_m-\psi_0}\\
&+\int_{M_j\times (D\setminus\{w_0\})}
|\mathbf{h}_{m,m',\epsilon,j}|^2_{p_2^*(h_{m'})} e^{-\phi-\tilde{\Psi}_m-\psi_0}\\
\le&
\int_{M_j\times (D\setminus\{w_0\})}
v''_{\epsilon}(\psi_0)|f|^2_{p_2^*(h_{m'})} e^{-\phi-\tilde{\Psi}_m-\psi_0}
<+\infty.
\end{split}
\end{equation}

It is clear that
\[u_{m,m',\epsilon,j}\in L^2(M_j\times D, K_{\Omega}\otimes E',p_2^*(h_{m'})e^{-\Phi_m}),\]
\[\mathbf{h}_{m,m',\epsilon,j}\in L^2(M_j\times D, \wedge^{n+1,1}T^*\Omega\otimes E', p_2^*(h_{m'})e^{-\Phi_m}).\]
In addition, it follows from inequality \eqref{estimate 1} that
\begin{equation}
\label{estimate 2}
\begin{split}
&\int_{M_j\times D}|u_{m,m',\epsilon,j}|^2_{p_2^*(h_{m'})}e^{v_\epsilon(\psi_0)-\tilde{\Psi}_m-\psi_0}\\
&+\int_{M_j\times D}|\mathbf{h}_{m,m',\epsilon,j}|^2_{p_2^*(h_{m'})}e^{-\phi-\tilde{\Psi}_m-\psi_0}\\
\leq&\int_{M_j\times D}
v''_{\epsilon}(\psi_0)|f|^2_{p_2^*(h_{m'})} e^{-\phi-\tilde{\Psi}_m-\psi_0}
<+\infty.
\end{split}
\end{equation}
By the construction of $v_{\epsilon}(t)$, we know $e^{v_{\epsilon}(\psi_0)}$ has a positive lower bound on $M_j\times D$. By the constructions of $v_{\epsilon}(t)$ and $u$, we know $e^{-\phi}=e^{-u(-v_{\epsilon}(\psi_0))}$ has a positive lower bound on $M_j\times D$. Also we know $e^{-\tilde{\Psi}_m}$ has a positive lower bound on $M_j\times D$. Note that $h_{m'}$ is $C^2$ smooth on $M_j\Subset M$. Hence by Lemma \ref{barpartialv=h} we have
\begin{equation}
\label{d-bar realation u,h,lamda 1}
D''u_{m,m',\epsilon,j}+P_{m,m'}\big(\sqrt{N_1 \hat{\lambda}_{m'}}\mathbf{h}_{m,m',\epsilon,j}\big)=\mu
\end{equation}
on $M_j\times D$.

\

\emph{Step 4: Letting $m'\to+\infty$. }

Note that $\sup_{M_j\times D}(e^{-u(-v_{\epsilon}(\psi_0))})<+\infty$,
$0\le
v_{\epsilon}{''}(t) \le \frac{2}{B}\mathbb{I}_{(-t_0-B+\epsilon,-t_0-\epsilon)}$ and $|e_x|_{h_{m'}}\le |e_x|_{h_{m'+1}}\le |e_x|_{h}$ for any $m'\in \mathbb{Z}_{\ge 0}$. We have
\begin{equation}
\begin{split}
\label{upper bound for rhs 1}
&v''_{\epsilon}(\psi_0)|f|^2_{p_2^*(h_{m'})} e^{-u(-v_{\epsilon}(\psi_0))-\psi_0}\\
\le &\sup_{M_j\times D}\left(e^{-u(-v_{\epsilon}(\psi_0))+t_0+B-\epsilon}\right)\frac{2}{B}\mathbb{I}_{\{-t_0-B+\epsilon<\psi_0<-t_0-\epsilon\}} |f|_{p_2^*(h)}^2.
\end{split}
\end{equation}

It follows from Lebesgue's dominated convergence theorem that
\begin{equation}\nonumber
\begin{split}
&\lim_{m'\to+\infty}\int_{M_j\times D}
v''_{\epsilon}(\psi_0)|f|^2_{p_2^*(h_{m'})} e^{-u(-v_{\epsilon}(\psi_0))-\tilde{\Psi}_m-\psi_0}\\
=&\int_{M_j\times D}v''_{\epsilon}(\psi_0)|f|^2_{p_2^*(h)} e^{ -u(-v_{\epsilon}(\psi_0))-\tilde{\Psi}_m-\psi_0}<+\infty,
\end{split}
\end{equation}
since $\tilde{\Psi}_m$ is bounded and $\int_{M_j\times D}|f|_{p_2^*(h)}^2<+\infty$.

It follows from $\inf_{M_j\times D}e^{-v_{\epsilon}(\psi_0)-\tilde{\Psi}_m}>0$, inequalities \eqref{estimate 2}, \eqref{upper bound for rhs 1} that
\[\sup_{m'} \int_{M_j\times D}
|u_{m,m',\epsilon,j}|^2_{p_2^*(h_{m'})}e^{-\psi_0}<+\infty.\]
As $|e_x|_{h_{m'}}\le |e_x|_{h_{m'+1}}$ for any $m'\in \mathbb{Z}_{\ge 0}$, for any fixed $i$, we have
\[\sup_{m'} \int_{M_j\times D}
|u_{m,m',\epsilon,j}|^2_{p_2^*(h_i)}e^{-\psi_0}<+\infty.\]
Especially letting $h_i=h_1$, since the closed unit ball of the Hilbert space is
weakly compact, we can extract a subsequence $u_{m,m'',\epsilon,j} $ weakly
convergent to $u_{m,\epsilon,j} $ in $L^2(M_j\times D,K_{\Omega}\otimes E', p_2^*(h_1)e^{-\psi_0})$ as $m''\to+\infty$.
It follows from Lemma \ref{weakly convergence} that $u_{m,m'',\epsilon,j}\sqrt{e^{v_{\epsilon}(\psi_0)-\tilde{\Psi}}}$ weakly converges to $u_{m,\epsilon,j}\sqrt{e^{v_{\epsilon}(\psi_0)-\tilde{\Psi}}}$ in $L^2(M_j\times D,K_{\Omega}\otimes E', p_2^*(h_1)e^{-\psi_0})$ as $m''\to+\infty$.

For fixed $i\in\mathbb{Z}_{\ge 0}$, as $h_1$ and $h_i$ are both $C^2$ smooth hermitian metrics on $M_j\subset\subset X$, we know that the two norms in $L^2(M_j\times D, K_{\Omega}\otimes E', p_2^*(h_1)e^{-\psi_0})$ and $ L^2(M_j\times D, K_{\Omega}\otimes E', p_2^*(h_i)e^{-\psi_0})$ are equivalent. Note that $\sup_{m''} \int_{M_j\times D}
|u_{m,m'',\epsilon,j}|^2_{p_2^*(h_i)}e^{-\psi_0}<+\infty$. Hence we know that $u_{m,m'',\epsilon,j}\sqrt{e^{v_{\epsilon}(\psi_0)-\tilde{\Psi}}}$ also weakly converges to $u_{m,\epsilon,j}\sqrt{e^{v_{\epsilon}(\psi_0)-\tilde{\Psi}}}$ in $L^2(M_j\times D, K_{\Omega}\otimes E', p_2^*(h_i)e^{-\psi_0})$ as $m''\to+\infty$ by Lemma \ref{weaklycon_twonorms}.

Then we have
\begin{equation}\nonumber
\begin{split}
&\int_{M_j\times D}|u_{m,\epsilon,j}|^2_{p_2^*(h_i)}e^{v_{\epsilon}(\psi_0)-\tilde{\Psi}_m-\psi_0}\\
\le& \liminf_{m''\to+\infty}\int_{M_j\times D}|u_{m,m'',\epsilon,j}|^2_{p_2^*(h_i)}e^{v_{\epsilon}(\psi_0)-\tilde{\Psi}_m-\psi_0}\\
\le& \liminf_{m''\to+\infty}\int_{M_j\times D}v''_{\epsilon}(\psi_0)|f|^2_{p_2^*(h_{m''})} e^{u(v_{\epsilon}(\psi_0))-\tilde{\Psi}_m-\psi_0}\\
\le&
\int_{M_j\times D}v''_{\epsilon}(\psi_0)|f|^2_{p_2^*(h)} e^{u(v_{\epsilon}(\psi_0))-\tilde{\Psi}_m-\psi_0}
<+\infty.
\end{split}
\end{equation}
Letting $i\to +\infty$, by monotone convergence theorem, we have
\begin{equation}
\begin{split}
\int_{M_j\times D}|u_{m,\epsilon,j}|^2_{p_2^*(h)}e^{v_{\epsilon}(\psi_0)-\tilde{\Psi}_m-\psi_0}
\le
\int_{M_j\times D}v''_{\epsilon}(\psi_0)|f|^2_{p_2^*(h)}e^{ -u(-v_{\epsilon}(\psi_0))-\tilde{\Psi}_m-\psi_0}
<+\infty.
\label{estimate for uej}
\end{split}
\end{equation}

It follows from $\inf_{M_j\times D}e^{ -u(-v_{\epsilon}(\psi_0))-\tilde{\Psi}_m}>0$, inequalities \eqref{estimate 2}, \eqref{upper bound for rhs 1} that
\[\sup_{m''} \int_{M_j\times D}
|\mathbf{h}_{m,m'',\epsilon,j}|^2_{p_2^*(h_{m''})}e^{-\psi_0}<+\infty.\]
As $|e_x|_{h_{m'}}\le |e_x|_{h_{m'+1}}$ for any $m'\in \mathbb{Z}_{\ge 0}$, we have
\[\sup_{m''} \int_{M_j\times D}
|\mathbf{h}_{m,m'',\epsilon,j}|^2_{p_2^*(h_1)}e^{-\psi_0}<+\infty.\]
Since the closed unit ball of the Hilbert space is
weakly compact, we can extract a subsequence of $\mathbf{h}_{m,m'',\epsilon,j} $ (also denoted by $\mathbf{h}_{m,m'',\epsilon,j} $) weakly
convergent to $\mathbf{h}_{m,\epsilon,j} $ in $L^2(M_j\times D,K_{\Omega}\otimes E', p_2^*(h_1)e^{-\psi_0})$ as $m''\to+\infty$. As $0\leq \hat{\lambda}_{m''}\leq \tilde{\lambda}_1+1$ and $M_j$ is relatively compact in $X$, we have
\[\sup_{m''} \int_{M_j\times D}
N_1\hat{\lambda}_{m''}|\mathbf{h}_{m,m'',\epsilon,j}|^2_{p_2^*(h_{m''})}e^{-\psi_0}<+\infty.\]
It follows from Lemma \ref{weakly converge lemma} that there exists a subsequence of $\{m''\}$ (also denoted by $\{m''\}$, such that $\sqrt{N_1\hat{\lambda}_{m''}} \mathbf{h}_{m,m'',\epsilon,j}$ is weakly convergent to some $\tilde{h}_{m,\epsilon,j}$ and  $P_{m''}\big(\sqrt{N_1\tilde{\lambda}_{m''}}\mathbf{h}_{m,m'',\epsilon,j} \big)$ weakly converges to $P\big(\tilde{h}_{m,\epsilon,j} \big)$ in $L^2(M_j\times D,\wedge^{n+1,1}T^*\Omega\otimes E', p_2^*(h_1)e^{-\psi_0})$.

It follows from  $0\le \hat{\lambda}_{m''}\le\tilde{\lambda}_1+1$, $\tilde{\lambda}_{m'}\to 0$, a.e., $M_j$ is relatively compact in $X$ and Lemma \ref{weakly convergence} that $\sqrt{N_1\tilde{\lambda}_{m''}}\mathbf{h}_{m,m'',\epsilon,j} $ weakly
convergent to $0$ in $L^2(M_j\times D,\wedge^{n+1,1}T^*\Omega\otimes E', p_2^*(h_1)e^{-\psi_0})$. It follows from the uniqueness of weak limit that $\tilde{h}_{m,\epsilon,j}=0$. Then we have $P_{m'}\big(\sqrt{N_1\hat{\lambda}_{m'}}\mathbf{h}_{m,m',\epsilon,j} \big)$ weakly converges to $0=P\big(\tilde{h}_{m,\epsilon,j} \big)$ in $L^2(M_j\times D,\wedge^{n+1,1}T^*\Omega\otimes E', p_2^*(h_1)e^{-\psi_0})$.

Replacing $m'$ by $m''$ in equality \eqref{d-bar realation u,h,lamda 1} and letting $m''$ go to $+\infty$, we have
\begin{equation}
\label{d-bar realation u,h,lamda 3_pre}
D'' u_{m,\epsilon,j}
=D''(1-v'_{\epsilon}(\psi_0))\wedge f,
\end{equation}
which implies that
\begin{equation}
\label{d-bar realation u,h,lamda 3}
D'' u_{m,\epsilon,j}
=D''\big((1-v'_{\epsilon}(\psi_0))f\big)
\end{equation}
on $M_j\times D$.

Denote $F_{m,\epsilon,j}:=-u_{m,\epsilon,j}+(1-v'_{\epsilon}(\psi_0))f$. It follows from equality \eqref{d-bar realation u,h,lamda 3} and inequality \eqref{estimate for uej} that we know $F_{m,\epsilon,j}$ is an $E'$-valued holomorphic $(n+1,0)$ form on $M_j\times D$ and
\begin{equation}
\label{estimate for F lej}
\begin{split}
&\int_{M_j\times D}
|F_{m,\epsilon,j}-(1-v'_{\epsilon}(\psi_0))f|^2_{p_2^*(h)} e^{v_\epsilon(\psi_0)-\tilde{\Psi}_m-\psi_0}\\
\le&\int_{M_j\times D}v''_{\epsilon}(\psi_0)|f|^2_{p_2^*(h)} e^{ -u(-v_{\epsilon}(\psi_0))-\tilde{\Psi}_m-\psi_0}<+\infty.
\end{split}
\end{equation}

\

\emph{Step 5: Letting $\epsilon\to 0$. }

Note that $\sup_{\epsilon}\sup_{M_j}(e^{-u(-v_{\epsilon}(\psi_0))-\tilde{\Psi}_m})<+\infty$,
$0\le
v_{\epsilon}{''}(t) \le \frac{2}{B}\mathbb{I}_{(-t_0-B+\epsilon,-t_0-\epsilon)}$. We have
\begin{equation}
	\label{upper bound for rhs 2 in L2 method}
	\begin{split}
	&v''_{\epsilon}(\psi_0)|f|^2_{p_2^*(h)} e^{ -u(-v_{\epsilon}(\psi_0))-\tilde{\Psi}_m-\psi_0}\\
	\le &\sup_{\epsilon}\sup_{M_j\times D}\left(e^{-u(-v_{\epsilon}(\psi_0))-\tilde{\Psi}_m}\right) \frac{2}{B}\mathbb{I}_{\{-t_0-B<\psi_0<-t_0\}} |f|_{p_2^*(h)}^2e^{-\psi_0}.
	\end{split}
\end{equation}
It follows from dominated convergence theorem that
\begin{equation}\nonumber
	\begin{split}
		&\lim_{\epsilon\to 0}\int_{M_j\times D}
		v''_{\epsilon}(\psi_0)|f|^2_{p_2^*(h)} e^{ -u(-v_{\epsilon}(\psi_0))-\tilde{\Psi}_m-\psi_0}\\
		=&\int_{M_j\times D}v_{\epsilon}''(\psi_0)|f|^2_{p_2^*(h)} e^{-u(-v(\psi_0))-\tilde{\Psi}_m-\psi_0}\\
		\le&\bigg(\sup_{M_j\times D}e^{-u(-v(\psi_0))}\bigg)\int_{M_j\times D}\frac{1}{B}\mathbb{I}_{\{-t_0-B<\psi_0<-t_0\}} |f|^2_{p_2^*(h)}e^{-\tilde{\Psi}_m-\psi_0}.
	\end{split}
\end{equation}

Combining with \[\inf_{\epsilon}\inf_{M_j\times D}e^{v_\epsilon(\psi_0)-\tilde{\Psi}_m-\psi_0}>0,\]
we have
\[\sup_{\epsilon}\int_{M_j\times D}|F_{m,\epsilon,j}-(1-v'_{\epsilon}(\psi_0))f|^2_{p_2^*(h)}<+\infty.\]
Note that
\begin{equation}\nonumber
	\begin{split}
		\sup_{\epsilon}\int_{M_j\times D}|(1-v'_{\epsilon}(\psi_0))f|_{p_2^*(h)}^2
		\le \int_{M_j\times D}|f|_{p_2^*(h)}^2<+\infty,
	\end{split}
\end{equation}
which implies that $\sup_{\epsilon}\int_{M_j\times D}|F_{m,\epsilon,j}|_{p_2^*(h)}^2<+\infty$.

Especially, we know
$\sup_{\epsilon}\int_{M_j\times D}|F_{m,\epsilon,j}|_{p_2^*(h_1)}^2<+\infty$. Note that $h_1$ is a $C^2$ hermitian metric on $E$, $M_j\subset\subset M$ and $F_{m,\epsilon,j}$ is an $E'$-valued holomorphic $(n+1,0)$ form on $M_j\times D$. Then there exists a subsequence of $\{F_{m,\epsilon,j}\}_{\epsilon}$ (also denoted by $\{F_{m,\epsilon,j}\}_{\epsilon}$) compactly convergent to an $E'$-valued holomorphic $(n+1,0)$ form $F_{m,j}$ on $M_j\times D$.

Then it follows from Fatou's lemma that we have
\begin{equation}
	\label{step epsilon to 0 formula 3}
	\begin{split}
		&\int_{M_j\times D}|F_{m,j}-(1-b(\psi_0))f|_{p_2^*(h)}^2e^{v(\psi_0)-\psi_0-\tilde{\Psi}_m}\\
		\le&\liminf_{\epsilon\to 0}\int_{M_j\times D}|F_{m,\epsilon,j}-(1-v'_{\epsilon}(\psi_0))f|_{p_2^*(h)}^2e^{v_{\epsilon}(\psi_0)-\tilde{\Psi}_m-\psi_0}\\
		\le&\limsup_{\epsilon\to 0}\int_{M_j\times D}|F_{m,\epsilon,j}-(1-v'_{\epsilon}(\psi_0))f|_{p_2^*(h)}^2e^{v_{\epsilon}(\psi_0)-\tilde{\Psi}_m-\psi_0}\\
		\le&\limsup_{\epsilon\to 0}\int_{M_j\times D}v_{\epsilon}''(\psi_0)|f|^2_{p_2^*(h)} e^{ -u(-v_{\epsilon}(\psi_0))-\tilde{\Psi}_m-\psi_0}\\
		\le&\left(\sup_{M_j\times D}e^{-u(-v(\psi_0))}\right)\int_{M_j\times D}\frac{1}{B}\mathbb{I}_{\{-t_0-B<\psi_0<-t_0\}}|f|^2_{p_2^*(h)}e^{-\tilde{\Psi}_m-\psi_0},
	\end{split}
\end{equation}
where $b(t)=b_{t_0}(t)=\frac{1}{B}\int^{t}_{-\infty}\mathbb{I}_{\{-t_0-B< s < -t_0\}}ds$,
$v(t)=v_{t_0}(t)=\int^{t}_{-t_0}b_{t_0}(s)ds-t_0$.

\

\emph{Step 6: ODE System.}

Now we want to find $\eta$ and $\phi$ such that
$(\eta+g^{-1})=e^{-v_\epsilon(\psi_0)}e^{-\phi}$.
As $\eta=s(-v_{\epsilon}(\psi_0))$ and $\phi=u(-v_{\epsilon}(\psi_0))$, we have
$(\eta+g^{-1})e^{v_\epsilon(\psi_0)}e^{\phi}=\left((s+\frac{s'^2}{u''s-s''})e^{-t}e^u\right)\circ(-v_\epsilon(\psi_0))$.\\

Summarizing the above discussion about $s$ and $u$, we are naturally led to a system of
ODEs:
\begin{equation}
	\begin{split}
		&1). \ \left(s+\frac{s'^2}{u''s-s''}\right)e^{u-t}=1,\\
		&2). \  s'-su'=1,
	\end{split}
	\label{ODE SYSTEM}
\end{equation}
where $t\in (0,+\infty)$.

We  solve the ODE system \eqref{ODE SYSTEM} and get
\[u(t)=-\log(1-e^{-t}), \ s(t)=\frac{t}{1-e^{-t}}-1.\]

It follows that $s\in C^{\infty}(0,+\infty)$ satisfies
$s>0$ and $u\in C^{\infty}(0,+\infty)$ satisfies
$u''s-s''>0$.

As $u(t)=-\log(1-e^{-t})$ is decreasing with respect to $t$, then
it follows from $0\geq v(t) \geq \max\{t,-t_0-B\} \geq -t_0-B$, for any $t<0$
that
\begin{equation}
	\begin{split}
		\sup\limits_{M\times D}e^{-u(-v(\psi_0))} \le
		\sup\limits_{t\in(0,t_0+1]}e^{-u(t)}=1-e^{-t_0-B}.
	\end{split}
\end{equation}

Combining with inequality \eqref{step epsilon to 0 formula 3}, we have

\begin{equation}\label{tildeFonXD}
	\begin{split}
		&\int_{M_j\times D}|F_{m,j}-(1-b(\psi_0))f|_{p_2^*(h)}^2e^{v(\psi_0)-\psi_0-\tilde{\Psi}_m}\\
		\le &\left(1-e^{-t_0-B}\right)\int_{M_j\times D}\frac{1}{B}\mathbb{I}_{\{-t_0-B<\psi_0<-t_0\}}|f|^2_{p_2^*(h)}e^{-\tilde{\Psi}_m}<+\infty.
	\end{split}
\end{equation}

\

\emph{Step 7: Letting $t_0\to +\infty$.}

According to inequality (\ref{tildeFonXD}), for any $t_0>0$, there exists some $F_{m,j,t_0}$ such that
\begin{equation}\label{tildeFt_0B1}
	\begin{split}
		&\int_{M_j\times D}|F_{m,j,t_0}-(1-b_{t_0}(\psi_0))f|_{p_2^*(h)}^2e^{v_{t_0}(\psi_0)-\psi_0-\tilde{\Psi}_m}\\
		\le &\int_{M_j\times D}\mathbb{I}_{\{-t_0-1<\psi_0<-t_0\}}|f|^2_{p_2^*(h)}e^{-\tilde{\Psi}_m-\psi_0}.
	\end{split}
\end{equation}
Here we let $B=1$, $b_{t_0}(t):=b_{t_0,1}(t)$, and $v_{t_0}(t):=v_{t_0,1}(t)$.

By direct calculation under the local case,  we can get that
\begin{equation}\label{RHSt0toinfty}
	\begin{split}
    &\lim_{t_0\to\infty}\int_{M_j\times D}\mathbb{I}_{\{-t_0-1<\psi_0<-t_0\}}|f|^2_{p_2^*(h)}e^{-\tilde{\Psi}_m-\psi_0}\\
    =&\int_{M_j\times \{w_0\}}|u|^2_{h}e^{-\tilde{\Psi}_m|_{M\times\{w_0\}}}\\
    &\cdot\lim_{t_0\to\infty}\int_D\mathbb{I}_{\{-t_0-1<2\log|w-w_0|-2\log r<-t_0\}}e^{-2\log|w-w_0|+2\log r}d\lambda_D\\
    \le&\pi r^2\int_M|u|^2_{h}e^{-\tilde{\Psi}_{w_0}}<+\infty,
	\end{split}
\end{equation}
since $\tilde{\Psi}_m$ is smooth and $\tilde{\Psi}_m\geq\tilde{\Psi}$, where $d\lambda_D$ is the Lebesgue measure on $D$. 

Note that $\psi_0=p_1^*(2\log|w-w_0|-2\log r)$, $\tilde{\Psi}_m$ is bounded on any $M_j\times D$, $h\ge h_{j,1}$ where $h_{j,1}$ is smooth on $M_j$, and $F_{m,j,t_0}$, $f$ are holomorphic on $M_j\times D$. Considering that $e^{-\psi_0}$ is not integrable near $M\times \{w_0\}$, we can find that $F_{m,j,t_0}|_{M_j\times\{w_0\}}=u\wedge dw$ by inequality (\ref{tildeFt_0B1}).

Inequality (\ref{tildeFt_0B1}) and inequality (\ref{RHSt0toinfty}) also imply that
\begin{equation}\label{tildeFt_0B1.5}
		\limsup_{t_0\to+\infty}\int_{M_j\times D}|F_{m,j,t_0}-(1-b_{t_0}(\psi_0))f|_{p_2^*(h)}^2e^{v_{t_0}(\psi_0)-\tilde{\Psi}_m-\psi_0}< +\infty.
\end{equation}

Note that $v_{t_0}(\psi_0)\geq \psi_0$, then we have
\begin{equation}\label{tildeFt_0B2}
		\limsup_{t_0\to\infty}\int_{M_j\times D}|F_{m,j,t_0}-(1-b_{t_0}(\psi_0))f|_{p_2^*(h)}^2e^{-\tilde{\Psi}_m}\le\pi r^2\int_M|u|^2_{h}e^{-\tilde{\Psi}_{w_0}}<+\infty.
\end{equation}
In addition, we have
\begin{equation}\nonumber
    |(1-b_{t_0}(\psi_0))f|^2_{p_2^*(h)}e^{-\tilde{\Psi}_m}\leq|f|^2_{p_2^*(h)}e^{-\tilde{\Psi}_m},
\end{equation}
where $\int_{M_j\times D}|f|^2_{p_2^*(h)}e^{-\tilde{\Psi}_m}<+\infty$ since $\tilde{\Psi}_m$ is bounded on $M_j\times D$. Then according to Lebesgue's dominated convergence theorem, we get
\begin{equation}\label{1-bto0}
    \lim_{t_0\to+\infty}\int_{M_j\times D}|(1-b_{t_0}(\psi_0))f|^2_{p_2^*(h)}e^{-\tilde{\Psi}_m}=0.
\end{equation}
Combining inequalities (\ref{tildeFt_0B2}) and (\ref{1-bto0}), we know $\int_{M_j\times D}|F_{m,j,t_0}|^2_{p_2^*(h)}$ is uniformly upper bounded with respect to $t_0\in (0,+\infty)$. Note that $h$ is locally lower bounded, thus we can find a subsequence of $\{F_{m,j,t_0}\}$ (also denoted by $\{F_{m,j,t_0}\}$ itself) compactly convergent to $\tilde{F}_{m,j}$, where $\tilde{F}_{m,j}$ is an $E'-$valued holomorphic $(n+1,0)$ form on $M_j\times D$. Then following from Fatou's lemma, inequalities (\ref{tildeFt_0B2}) and (\ref{1-bto0}), we have
\begin{equation}\label{tildeFwitht0to+infty}
	\begin{split}
    &\int_{M_j\times D}|\tilde{F}_{m,j}|^2_{p_2^*(h)}e^{-\tilde{\Psi}_m}\\
    \le &\liminf_{t_0\to +\infty}\int_{M_j\times D}|F_{m,j,t_0}|^2_{p_2^*(h)}e^{-\tilde{\Psi}_m}\\
    \le &\limsup_{t_0\to\infty}\int_{M_j\times D}|F_{m,j,t_0}-(1-b_{t_0}(\psi_0))f|_{p_2^*(h)}^2e^{-\tilde{\Psi}_m}\\
    \le &\pi r^2\int_M|u|^2_{h}e^{-\tilde{\Psi}_{w_0}}<+\infty.
	\end{split}
\end{equation}
In addition, we have $\tilde{F}_{m,j}|_{M_j\times\{w_0\}}=u\wedge dw$, since $F_{m,j,t_0}|_{M_j\times\{w_0\}}=u\wedge dw$ for any $t_0$.

\

\emph{Step 8: Letting $m\to +\infty$.}

Since $\tilde{\Psi}_m$ is uniformly bounded on any $M_j\times D$ for any fixed $j$, we know that $\int_{M_j\times D}|\tilde{F}_{m,j}|^2_{p_2^*(h)}$ is uniformly bounded with respect to $m$ by inequality (\ref{tildeFwitht0to+infty}). Note that $h$ is locally lower bounded, thus we can find a subsequence of $\{\tilde{F}_{m,j}\}$ (also denoted by $\{\tilde{F}_{m,j}\}$ itself) compactly convergent to $\tilde{F}_{j}$, where $\tilde{F}_{j}$ is an $E'-$valued holomorphic $(n+1,0)$ form on $M_j\times D$. According to Fatou's Lemma, we have
\begin{equation}\label{tildeFwithmto+infty}
	\begin{split}
    &\int_{M_j\times D}|\tilde{F}_{j}|^2_{p_2^*(h)}e^{-\tilde{\Psi}}\\
    \le &\liminf_{m\to +\infty}\int_{M_j\times D}|\tilde{F}_{m,j}|^2_{p_2^*(h)}e^{-\tilde{\Psi}_m}\\
    \le &\pi r^2\int_M|u|^2_{h}e^{-\tilde{\Psi}_{w_0}}<+\infty.
	\end{split}
\end{equation}
Additionally, we have $\tilde{F}_{j}|_{M_j\times\{w_0\}}=u\wedge dw$.

\

\emph{Step 9: Letting $j\to +\infty$.}

Since $\tilde{\Psi}$ is bounded on $\Omega=M\times D$, we have $\int_{M_j\times D}|\tilde{F}_{j}|^2_{p_2^*(h)}$ is uniformly bounded with respect to $j$ by inequality (\ref{tildeFwithmto+infty}). Note that $h$ is locally lower bounded, thus by diagonal method we can find a subsequence of $\{\tilde{F}_{j}\}$ (also denoted by $\{\tilde{F}_{j}\}$ itself) convergent to $\tilde{u}$ on any $M_j\times D$, where $\tilde{u}$ is an $E'-$valued holomorphic $(n+1,0)$ form on $\Omega$. Then it follows from Fatou's Lemma that
\begin{equation}\label{tildeFwithjto+infty}
	\begin{split}
    &\int_{M\times D}|\tilde{u}|^2_{p_2^*(h)}e^{-\tilde{\Psi}}\\
    \le &\liminf_{j\to +\infty}\int_{M_j\times D}|\tilde{F}_{j}|^2_{p_2^*(h)}e^{-\tilde{\Psi}}\\
    \le &\pi r^2\int_M|u|^2_{h}e^{-\tilde{\Psi}_{w_0}}.
	\end{split}
\end{equation}
According to $\tilde{F}_{j}|_{M_j\times\{w_0\}}=u\wedge dw$, we also have $\tilde{u}|_{M_j\times\{w_0\}}=u\wedge dw$, which $\tilde{u}$ is actually the extension of $u$ what we need.

Then the proof of Lemma \ref{L2ext} is done.

\vspace{.1in} {\em Acknowledgements}. We would like to thank Dr. Zhitong Mi and Zheng Yuan for checking this paper. The second author was supported by National Key R\&D Program of China 2021YFA1003100, NSFC-11825101, NSFC-11522101 and NSFC-11431013.

\bibliographystyle{references}
\bibliography{xbib}

\begin{thebibliography}{100}
	
	\bibitem{BG-BB}
	S.J. Bao and Q.A. Guan, Modules at boundary points, fiberwise Bergman kernels, and log-subharmonicity \uppercase\expandafter{\romannumeral2} -- on Stein manifolds, arXiv:2205.08044 [math.CV].
	
	\bibitem{BG-SM}
	S.J. Bao and Q.A. Guan, Modules at boundary points, fiberwise Bergman kernels, and log-subharmonicity, arXiv:2204.01413 [math.CV], to appear in Peking Mathematical Journal.
	
	\bibitem{BGY}
	S.J. Bao, Q.A. Guan and Z. Yuan, Boundary points, minimal $L^2$ integrals and concavity property, arXiv:2203.01648.v2 [math.CV].
	
	\bibitem{berndtsson20}B. Berndtsson, Lelong numbers and vector bundles, J. Geom. Anal. 30 (2020), no. 3, 2361-2376.
	
	\bibitem{cao17}J.Y. Cao, Ohsawa-Takegoshi extension theorem for compact K{\"a}hler manifolds and applications, Complex and symplectic geometry, 19-38, Springer INdAM Ser., 21, Springer, Cham, 2017.
	
	\bibitem{cdM17}J.Y. Cao, J-P. Demailly and S. Matsumura, A general extension theorem for cohomology classes on non reduced analytic subspaces, Sci. China Math. 60 (2017), no. 6, 949-962, DOI 10.1007/s11425-017-9066-0.
	
	\bibitem{DEL18}T. Darvas, E. Di Nezza and H.C. Lu, Monotonicity of nonpluripolar products and complex Monge-Amp{\'e}re equations with prescribed singularity, Anal. PDE 11 (2018), no. 8, 2049-2087.
	
	\bibitem{DEL21}T. Darvas, E. Di Nezza and H.C. Lu, The metric geometry of singularity types, J. Reine Angew. Math. 771 (2021), 137-170.
	
	\bibitem{Demailly82}J.-P Demailly, Estimations $L^2$ pour l'op\'erateur $\bar\partial$ d'un fibr\'e vectoriel holomorphe semi-positif au-dessus d'une vari\'et\'e k\"ahl\'erienne compl\`ete.(French) \textit{$L^2$ estimates for the $\bar\partial$-operator of a semipositive
		holomorphic vector bundle over a complete K\"ahler manifold}, Ann. Sci. \'Ecole Norm. Sup. (4) 15 (1982), no. 3, 457-511.
	
	\bibitem{Demaillybook}J.-P Demailly, Complex analytic and differential geometry, electronically accessible at https://www-fourier.ujf-grenoble.fr/\textasciitilde demailly/manuscripts/agbook.pdf.
	
	\bibitem{DemaillySoc}J.-P Demailly, Multiplier ideal sheaves and analytic methods in algebraic geometry, School on Vanishing Theorems and Effective Result in Algebraic Geometry (Trieste,2000),1-148,ICTP lECT.Notes, 6, Abdus Salam Int. Cent. Theoret. Phys., Trieste, 2001.
	
	\bibitem{DemaillyAG}J.-P Demailly, Analytic Methods in Algebraic Geometry, Higher Education Press, Beijing, 2010.
	
	\bibitem{DEL}J.-P Demailly, L. Ein and R. Lazarsfeld, A subadditivity property of multiplier ideals, Michigan Math. J. 48 (2000) 137-156.
	
	\bibitem{DK01}J.-P Demailly and J. Koll\'ar, Semi-continuity of complex singularity exponents and K\"ahler-Einstein metrics on Fano orbifolds, Ann. Sci. \'Ec. Norm. Sup\'er. (4) 34 (4) (2001) 525-556.
	
	\bibitem{DP03}J.-P Demailly and T. Peternell, A Kawamata-Viehweg vanishing theorem on compact K\"ahler manifolds, J. Differential Geom. 63 (2) (2003) 231-277.
	
	\bibitem{FN80}
	J. E. Forn{\ae}ss, R. Narasimhan:
	The {L}evi problem on complex spaces with singularities. Math. Ann., 248 (1980), no. 1, 47-72.
	
	\bibitem{FoW18}J.E. Forn{\ae}ss and J.J.  Wu, A global approximation result by Bert Alan Taylor and the strong openness conjecture in $\mathbb{C}^n$, J. Geom. Anal. 28 (2018), no. 1, 1-12.
	
	\bibitem{FoW20}J.E. Forn{\ae}ss and J.J.  Wu, Weighted approximation in $\mathbb{C}$, Math. Z. 294 (2020), no. 3-4, 1051-1064.
	
	\bibitem{GMY-BC2} Q.A. Guan, Z.T. Mi and Z. Yuan, Boundary points, minimal $L^2$ integrals and concavity property \uppercase\expandafter{\romannumeral2}: on weakly pseudoconvex K\"{a}hler manifold, arXiv:2203.07723.v2 [math.CV].

	\bibitem{GMY5}Q.A. Guan, Z.T. Mi and Z. Yuan, Boundary points, minimal $L^2$ integrals and concavity property \uppercase\expandafter{\romannumeral5}---vector bundles, arXiv:2206.00443 [math.CV].
	
	\bibitem{GMY-L2ext}Q.A. Guan, Z.T. Mi and Z. Yuan, Optimal $L^2$ extension for holomorphic vector bundles with singular hermitian metrics, arXiv:2210.06026 [math.CV].

	\bibitem{guan-zhou13ap}Q.A. Guan and X.Y. Zhou, A solution of an $L^{2}$ extension problem with an optimal estimate and applications, Ann. of Math. (2) 181 (2015), no. 3, 1139--1208.
	
	\bibitem{GZSOC}Q.A. Guan and X.Y Zhou, A proof of Demailly's strong openness conjecture, Ann. of Math. (2) 182 (2015), no. 2, 605-616.
	
	\bibitem{GZeff}Q.A. Guan and X.Y Zhou, Effectiveness of Demailly's strong openness conjecture and related problems, Invent. Math. 202 (2015), no. 2, 635-676.

	\bibitem{GZ20}Q.A. Guan and X.Y. Zhou, Restriction formula and subadditivity property related to multiplier ideal sheaves, J. Reine Angew. Math. 769, 1-33 (2020).
	
	\bibitem{Guenancia}H. Guenancia, Toric plurisubharmonic functions and analytic adjoint ideal sheaves, Math. Z. 271 (3-4) (2012) 1011-1035.

	\bibitem{JM12}M. Jonsson and M. Musta\c{t}\u{a}, Valuations and asymptotic invariants for sequences of ideals, Annales de l'Institut Fourier A. 2012, vol. 62, no.6, pp. 2145--2209.
	
	\bibitem{JM13}M. Jonsson and M. Musta\c{t}\u{a}, An algebraic approach to the openness conjecture of Demailly and Koll\'{a}r, J. Inst. Math. Jussieu (2013), 1--26.		
	
	\bibitem{K16}D. Kim, Skoda division of line bundle sections and pseudo-division, Internat. J. Math. 27 (2016), no. 5, 1650042, 12 pp.
	
	\bibitem{KS20}D. Kim and H. Seo, Jumping numbers of analytic multiplier ideals (with an appendix by Sebastien Boucksom), Ann. Polon. Math., 124 (2020), 257-280.
	
	\bibitem{Lazarsfeld}R. Lazarsfeld,
	Positivity in Algebraic Geometry. \uppercase\expandafter{\romannumeral1}. Classical Setting: Line Bundles and Linear Series. Ergebnisse der Mathematik und ihrer Grenzgebiete. 3. Folge. A Series of Modern Surveys in Mathematics [Results in Mathematics and Related Areas. 3rd Series. A Series of Modern Surveys in Mathematics], 48. Springer-Verlag, Berlin, 2004;\\
	R. Lazarsfeld,
	Positivity in Algebraic Geometry. \uppercase\expandafter{\romannumeral2}. Positivity for vector bundles, and multiplier ideals. Ergebnisse der Mathematik und ihrer Grenzgebiete. 3. Folge. A Series of Modern Surveys in Mathematics [Results in Mathematics and Related Areas. 3rd Series. A Series of Modern Surveys in Mathematics], 49. Springer-Verlag, Berlin, 2004.

	\bibitem{Nadel}A. Nadel, Multiplier ideal sheaves and K\"ahler-Einstein metrics of positive scalar curvature, Ann. of Math. (2) 132 (3) (1990) 549-596.
	
	\bibitem{Siu96}Y.T. Siu, The Fujita conjecture and the extension theorem of Ohsawa-Takegoshi, Geometric Complex Analysis, World Scientific, Hayama, 1996, pp.223-277.
	
	\bibitem{Siu05}Y.T. Siu, Multiplier ideal sheaves in complex and algebraic geometry, Sci. China Ser. A 48 (suppl.) (2005) 1-31.
	
	\bibitem{Siu09}Y.T. Siu, Dynamic multiplier ideal sheaves and the construction of rational curves in Fano manifolds, Complex Analysis and Digtial Geometry, in: Acta Univ. Upsaliensis Skr. Uppsala Univ. C Organ. Hist., vol.86, Uppsala Universitet, Uppsala, 2009, pp.323-360.
	
	\bibitem{Tian}G. Tian, On K\"ahler-Einstein metrics on certain K\"ahler manifolds with $C_1(M)>0$, Invent. Math. 89 (2) (1987) 225-246.
	
	\bibitem{ZZ2018}X.Y Zhou and L.F.Zhu,An optimal $L^2$ extension theorem on weakly pseudoconvex K\"ahler manifolds, J. Differential Geom.110(2018), no.1, 135-186.
	
	
	\bibitem{ZZ2019}X.Y Zhou and L.F.Zhu, Optimal $L^2$ extension of sections from subvarieties in weakly pseudoconvex manifolds. Pacific J. Math. 309 (2020), no. 2, 475-510.
	
	\bibitem{ZhouZhu20siu's}X.Y. Zhou and L.F. Zhu, Siu's lemma, optimal $L^2$ extension and applications to twisted pluricanonical sheaves, Math. Ann. 377 (2020), no. 1-2, 675-722.
	
	
\end{thebibliography}

\end{document}